\documentclass[letterpaper,12pt]{amsart}

\usepackage{amssymb,amsmath,setspace,graphicx,multicol,
wrapfig,amsfonts,amsthm,url,array,parskip,enumitem,bbm,color,
tikz,mathabx}
\usepackage{hyperref}
\usetikzlibrary{decorations.pathreplacing,calligraphy}
\usepackage{colortbl}
\usepackage{multirow}

\usepackage{bm}
\usepackage{color,float}
\usepackage[capitalize,nameinlink,noabbrev,nosort]{cleveref}
\usepackage{ytableau}
\usepackage{placeins}
\usepackage[hmargin=2.5cm,vmargin=2.5cm]{geometry}
\makeatletter
\def\l@subsection{\@tocline{2}{0pt}{2.5pc}{5pc}{}}
\makeatother

\numberwithin{equation}{section}

\theoremstyle{plain}
\newtheorem{thm}{Theorem}[section]
\newtheorem{lemma}[thm]{Lemma}
\newtheorem{prop}[thm]{Proposition}
\newtheorem{cor}[thm]{Corollary}
\newtheorem{example}[thm]{Example}
\crefname{lemma}{Lemma}{Lemmas}
\crefname{thm}{Theorem}{Theorems}

\theoremstyle{definition}
\newtheorem{defn}[thm]{Definition}

\newtheorem{conj}[thm]{Conjecture}

\crefname{conj}{Conjecture}{Conjectures}

\theoremstyle{remark}
\newtheorem*{remark}{Remark}

\DeclareMathOperator{\In}{In}
\DeclareMathOperator{\Out}{Out}

\DeclareMathOperator{\wt}{wt}
\DeclareMathOperator{\quinv}{quinv}
\DeclareMathOperator{\coinv}{coinv}

\DeclareMathOperator{\South}{South}
\DeclareMathOperator{\North}{North}
\DeclareMathOperator{\arm}{LArm}
\DeclareMathOperator{\uarm}{UArm}
\DeclareMathOperator{\up}{up}
\DeclareMathOperator{\down}{down}
\DeclareMathOperator{\dg}{dg}
\DeclareMathOperator{\Tab}{Tab}
\DeclareMathOperator{\Ext}{Ext}
\DeclareMathOperator{\TAZRP}{\mathcal{T}}
\DeclareMathOperator{\rate}{rate}
\DeclareMathOperator{\inv}{inv}
\DeclareMathOperator{\cont}{cont}
\DeclareMathOperator{\proj}{\mathcal{P}}
\DeclareMathOperator{\lambdac}{\lambda^{\text{c}}}

\newcommand{\bI}{\bm{I}}
\newcommand{\qbinom}[2]{\bgroup\renewcommand*{\arraystretch}{1}\begin{bmatrix} #1 \\ #2\end{bmatrix} \egroup}
\DeclareMathOperator{\maj}{maj}
\DeclareMathOperator{\sh}{\mathcal{L}}
\DeclareMathOperator{\A}{a}

\DeclareMathOperator{\LLT}{LLT}
\DeclareMathOperator{\sle}{\sigma^{\leq}}
\DeclareMathOperator{\sg}{\sigma^{>}}

\newcommand{\NN}{\mathbb{N}}
\newcommand{\cQ}{\mathcal{Q}}
\newcommand{\cM}{\mathcal{M}}

\newcommand{\tS}{\tilde{S}}
\newcommand{\tx}{\tilde{x}}

\newcommand{\PP}{\mathbb{P}}

\newlength\cellsize 
\savebox2{%
\begin{picture}(12,12)
\put(0,0){\line(1,0){12}}
\put(0,0){\line(0,1){12}}
\put(12,0){\line(0,1){12}}
\put(0,12){\line(1,0){12}}
\end{picture}}

\newcommand\cellify[1]{\def\thearg{#1}\def\nothing{}%
\ifx\thearg\nothing
\vrule width0pt height\cellsize depth0pt\else
\hbox to 0pt{\usebox2\hss}\fi%
\vbox to 12\unitlength{
\vss
\hbox to 12\unitlength{\hss$#1$\hss}
\vss}}

\newcommand\tableau[1]{\vtop{\let\\=\cr
\setlength\baselineskip{-16000pt}
\setlength\lineskiplimit{16000pt}
\setlength\lineskip{0pt}
\halign{&\cellify{##}\cr#1\crcr}}}

\savebox3{%
\begin{picture}(12,12)
\put(0,0){\line(1,0){12}}
\put(0,0){\line(0,1){12}}
\put(12,0){\line(0,1){12}}
\put(0,12){\line(1,0){12}}
\end{picture}}

\newcommand\expath[1]{%
\hbox to 0pt{\usebox3\hss}%
\vbox to 12\unitlength{
\vss
\hbox to 12\unitlength{\hss$#1$\hss}
\vss}}

\newcommand\cell[3]{
\def\i{#1} \def\j{#2} \def\entry{#3}

\draw (\j-1,-\i)--(\j,-\i)--(\j,-\i+1)--(\j-1,-\i+1)--(\j-1,-\i);
\node at (\j-.5,-\i+.5) {\entry};
}

\newcommand\greysq{
\begin{tikzpicture}[scale=0.5]
\def\h{0.85};
\filldraw[white,fill=gray!25] (-1*\h,0)--(0,0)--(0,1*\h)--(-1*\h,1*\h)--(-1*\h,0);
\end{tikzpicture}
}
\newcommand\darkgreysq{
\begin{tikzpicture}[scale=0.5]
\def\h{0.85};
\filldraw[white,fill=gray!60] (-1*\h,0)--(0,0)--(0,1*\h)--(-1*\h,1*\h)--(-1*\h,0);
\end{tikzpicture}
}

\newcommand{\qtrip}[3]{
\begin{tikzpicture}[scale=0.5]
\cell{1}{0}{#1} \cell{2}{0}{#2} \cell{2}{2.7}{#3}
\node at (1,-1.5) {$\cdots$};
\end{tikzpicture}
}

\newcommand\cellL[4]{
\def\i{#1} \def\j{#2} \def\entry{#3} \def\labl{#4}
\draw (\j-1,-\i)--(\j,-\i)--(\j,-\i+1)--(\j-1,-\i+1)--(\j-1,-\i);
\node at (\j-.5,-\i+.5) {\entry};
\node at (\j-1.3,-\i+.8) {\tiny\labl};
}

\begin{document}

\title{Modified {M}acdonald polynomials and the multispecies zero range process: II}

\author[A.~Ayyer]{Arvind Ayyer}
\address{A.~Ayyer, Department of Mathematics, Indian Institute of Science, Bangalore 560 012, India}
\email{arvind@iisc.ac.in}

\author[O.~Mandelshtam]{Olya Mandelshtam}\address{O.~Mandelshtam, Department of Combinatorics and Optimization, University of Waterloo, Waterloo, ON, Canada}
\email{omandels@uwaterloo.ca}

\author[J.~B.~Martin]{James B.\ Martin}
\address{J.~B.~Martin, Department of Statistics, University of Oxford, UK}
\email{martin@stats.ox.ac.uk}

\date{\today}

\begin{abstract}
In a previous part of this work, we gave a new tableau formula for the modified Macdonald polynomials $\widetilde{H}_{\lambda}(X;q,t)$, using a weight on tableaux involving the \emph{queue inversion} (quinv) statistic. In this paper we explicitly describe a connection between these combinatorial objects and a class of multispecies totally asymmetric zero range processes (mTAZRP) on a ring, with site-dependent jump-rates. We construct a Markov chain on the space of tableaux of a given shape, which projects to the mTAZRP, and whose stationary distribution can be expressed in terms of quinv-weighted tableaux. We deduce that the mTAZRP has a partition function given by the modified Macdonald polynomial $\widetilde{H}_{\lambda}(X;1,t)$. The novelty here in comparison to previous works relating the stationary distribution of integrable systems to symmetric functions is that the variables $x_1,\ldots,x_n$ are explicitly present as hopping rates in the mTAZRP. We also obtain interesting symmetry properties of the mTAZRP probabilities under permutation of the jump-rates between the sites. Finally, we explore a number of interesting special cases of the mTAZRP, and give explicit formulas for particle densities and correlations of the process purely in terms of modified Macdonald polynomials.
\end{abstract}

\maketitle

\setcounter{tocdepth}{1}
\tableofcontents

\section{Introduction}
Over the past several years, there has been a growing body of research devoted to exploring the relationships of particle models and integrable systems with symmetric (and associated nonsymmetric) functions; see~\cite{A21} for many examples. These connections have led to new combinatorial objects, new positive formulas~\cite{compactformula,CMW18,garbali-wheeler-2020,AAAM22+}, and paths to deeper understanding of both the symmetric functions and the associated particle processes. One such connection, coming from \cite{CGW-2015}, was the remarkable discovery that the symmetric Macdonald polynomial $P_{\lambda}(X;q,t)$ specializes to the partition function of the multispecies asymmetric simple exclusion process (ASEP) on a ring.

The family of Macdonald polynomials $P_{\lambda}(X;q,t)$ (in the variables $X=x_1,x_2,\ldots$ and with coefficients in $\mathbb{Q}(q,t)$), originally introduced by Macdonald \cite{Mac88}, is one of the most celebrated families of symmetric functions. The Macdonald polynomials are orthogonal symmetric polynomials defined by certain triangularity and orthogonality axioms, which notably contain as special cases the Schur functions, Hall-Littlewood polynomials, Jack polynomials, and $q$-Whittaker functions.  

The \emph{modified Macdonald polynomials} $\widetilde{H}_{\lambda}(X;q,t)$, introduced by Garsia and Haiman \cite{GarsiaHaiman96}, are a combinatorial version of the $P_{\lambda}$'s, obtained via plethysm from the integral form of the latter. These polynomials enjoy some particularly nice properties such as having positive integer coefficients and being Schur-positive. Following the discovery of the link between Macdonald polynomials $P_{\lambda}(X;q,t)$ and the partition function of the multispecies ASEP, a natural question was whether there exists a related statistical mechanics model for which some specialization of $\widetilde{H}_{\lambda}(X;q,t)$ is equal to its partition function. In this article, we describe such a model, which turns out to be a \emph{multispecies totally asymmetric zero range process} (mTAZRP), in a form which had previously been considered in \cite{takeyama-2015}. 
 
The zero range process (ZRP) is an important example of a non-equilibrium exactly-solvable interacting particle system. It was introduced by Spitzer in an influential paper that initiated the mathematical study of interacting particle systems in 1970~\cite{spitzer-1970}, along with the related simple exclusion process. The ZRP describes particles hopping from one site of a graph to another with a rate that depends solely on the content of the departure site and is independent of the destination site (hence the name `zero range'). Among the many reasons that the ZRP is of great interest to physicists is that it is an excellent toy model for the physics of phase-separation and condensation; see~\cite{EvansHanney} for a review. A \emph{totally asymmetric} ZRP (TAZRP) is a ZRP in which particles can only hop in one direction. In our setting, the underlying lattice is a ring (i.e.~a line with periodic boundary conditions).

We consider a multispecies TAZRP that has particles of different classes (or species), that is related to a wider class of multispecies zero range processes, still with integrable properties; see the review in \cite{kuniba-okado-watanabe-2017}. The species of particles in the system are determined by a partition $\lambda=(\lambda_1, \lambda_2,\ldots, \lambda_k)$ where $\lambda_1\geq \lambda_2\dots\geq \lambda_k>0$.
Such a system has $k$ particles in all, of species $\lambda_1, \lambda_2, \dots, \lambda_k$. Transitions in the mTAZRP on $n$ sites involve jumps of a single particle from a site to its clockwise neighbour with a rate in terms of parameters $t,x_1,\ldots,x_n$ that depends on the site, the species of the jumping particle, and other particles at that site. The total jump rate at site $j$, if it contains $m$ particles, is
$x_j^{-1} \, (1-t^m)/(1-t)$. See \cref{sec:tazrp} for the precise definition of the model.

In this work we study the stationary distribution of the mTAZRP. Single species zero range processes are notable for having stationary distributions of product form. However, the stationary distributions of multispecies systems have a much richer structure. Our approach involves considering fillings of the diagram of the partition $\lambda$ with entries in $\{1,\ldots, n\}$ -- see, for example, \cref{fig:polyqueue}. Let $\Tab(\lambda,n)$ denote the set of such fillings. In part I of this work \cite{AMM20} we introduced the ``queue inversion" statistic $\quinv:\Tab(\lambda,n)\mapsto\NN$ and used it to obtain a new formula for the modified Macdonald polynomial. Then
\[
\widetilde{H}_{\lambda}(x_1,\dots,x_n;1,t)
=
\sum_{\sigma\in\Tab(\lambda,n)}x^{\sigma}t^{\quinv(\sigma)}.
\]

We define a function $\proj$ from $\Tab(\lambda,n)$ to the configurations of the mTAZRP. 

The first main result in this article is that the stationary distribution of the mTAZRP on $n$ sites with particle types given by a partition $\lambda$ is the projection under the map $\proj$ of the distribution of the $\quinv$ statistic on $\Tab(\lambda)$:
\begin{thm}\label{thm:TAZRP-stat-distn}
Consider the mTAZRP on $n$ sites with particle species given by the partition $\lambda$. Then the stationary probability of a configuration $w$ in that mTAZRP equals
\[
\frac{1}{\mathcal{Z}_{(\lambda,n)}} \sum_{\substack{\sigma\in\Tab(\lambda,n)
\\
\proj(\sigma)=w}} x^{\sigma}t^{\quinv(\sigma)},
\] 
where $\mathcal{Z}_{(\lambda,n)}=\widetilde{H}_{\lambda}(x_1,\dots,x_n;1,t)$.
\end{thm}

In this sense, the modified Macdonald polynomial $\widetilde{H}_{\lambda}$ specializes at $q=1$ to the \emph{partition function} for the mTAZRP model: all the stationary probabilities may be written as rational functions in $\mathbb{N}[t;x_1,\ldots,x_n]$ with denominator $\widetilde{H}_{\lambda}(x_1,\ldots,x_n;1,t)$.

All the objects outlined above will be defined formally in \cref{sec:tazrp,sec:Tab}. 

We prove \cref{thm:TAZRP-stat-distn} by constructing a Markov chain on the state space $\Tab(\lambda, n)$, which has two key properties:
\begin{itemize}
\item[(i)] it projects to the mTAZRP via the function $\proj$
\item[(ii)] the stationary probability of a filling $\sigma$ is  proportional to the weight $x^{\sigma}t^{\quinv(\sigma)}$. 
\end{itemize}

Our construction is very closely related to the approach involving multiline queues, which was first introduced for the TASEP~\cite{FM07} and has subsequently been used to describe the stationary distributions of a variety of multispecies interacting particle systems~\cite{FerrariMartinHAD, martin-2020,CMW18,AAAM22+}  and related objects such as Busemann functions for last-passage percolation~\cite{FanSeppalainen}. Particular applications include the description of “speed processes” and the computation of higher correlations~\cite{amir-angel-valko-2011,ayyer-linusson-2016}. There are close connections to the matrix product representation of the stationary distribution \cite{evans-ferrari-mallick-2009,prolhac-evans-mallick-2009,arita-mallick-2013,ayyer-linusson-2014}. The constructions in this paper could also be written in terms of multiline queues, but that picture is made more complicated by the need to keep track of the motion of several different particles at the same site. In this setting, the tableau construction provides a particularly convenient way to represent the data and to describe the dynamics.

A key point of interest is the presence of the parameters $x_1, x_2, \dots, x_n$ representing site-dependent rates in the mTAZRP. In contrast, for the ASEP for example, existing results cover only the case where all rates are equal, involving the specialization of the Macdonald polynomials to $x_1=\cdots=x_n=1$. By exploiting symmetries of fillings, we obtain interesting new symmetry properties for the mTAZRP under permutation of the state-dependent jump rates:

\begin{thm}\label{thm:TAZRP-symmetry}
Consider the stationary distribution of the mTAZRP on $n$ sites with particle species given by $\lambda$, and parameters $x_1, \dots, x_n, t$. 
The distribution of the configuration restricted to sites $1, \dots, \ell$ is symmetric in the variables $\{x_{\ell+1}, \dots, x_n\}$.
\end{thm}

In the special case $t=0$, a stronger result can be proved involving \textit{pathwise} symmetry properties of the mTAZRP process, rather than those involving a single configuration drawn from the stationary distribution -- see Theorem \ref{thm:stronger-symmetry}. The proof involves interchangeability properties of queueing servers with different rates. It is an interesting open problem whether such pathwise symmetry properties also hold for $t>0$. 

The article is organized as follows. The mTAZRP is formally defined in \cref{sec:tazrp}. The necessary background on fillings of diagrams and tableaux statistics is given in \cref{sec:Tab}. \cref{sec:Markov} contains the definition of the Markov chain on $\Tab(\lambda,n)$ and the proof of our main results, modulo some technical results that are postponed to \cref{sec:proofs}. \cref{sec:special} gives nice formulas for two special cases of the mTAZRP: the single species case and the $t=0$ case. In \cref{sec:partition}, we study the partition function of the mTAZRP and its properties. In Section \ref{sec:symmetry} and Section \ref{sec:stronger-symmetry}, we study the symmetry properties of the mTAZRP and prove Theorem \ref{thm:TAZRP-symmetry}, along with its extension in the case $t=0$. In Section \ref{sec:dens-curr}, we obtain formulas for single-site densities and currents of particles of a given species. In Section \ref{sec:top}, we discuss an interesting consistency property of the tableau process. Finally, the connection of the tableaux process with multiline queues is explained in \cref{sec:multiline}.

\bigskip

\noindent{\bf Acknowledgements:~} 
We would like to thank the anonymous referee for useful comments. AA was partially supported by SERB Core grant CRG/2021/001592 as well as the DST FIST program - 2021 [TPN - 700661]. OM was partially supported by NSERC grant RGPIN-2021-02568 and NSF grant DMS-1953891.

\section{Definition of the mTAZRP}
\label{sec:tazrp}

A \emph{partition} $\lambda = (\lambda_1, \dots, \lambda_k)$ is a weakly decreasing sequence of positive integers. It will be convenient at times to think of a partition as an infinite sequence with infinitely many trailing zeros. We will sometimes denote partitions using the frequency notation $\lambda = \langle 1^{m_1},\dots,j^{m_j},\dots \rangle$, which means that the part $j$ occurs $m_j$ times for $j=1,2,\ldots$. When all multiplicities $m_j$ are equal to 0 or 1
(i.e.\ when all the $\lambda_i$ are distinct),
$\lambda$ is called a \emph{strict partition}. 
The configurations of the mTAZRP are determined by a partition $\lambda$ 
and a positive integer $n$. 

We consider a one-dimensional lattice with $n$ sites (labelled $1, 2, \ldots,n$) with periodic boundary conditions, containing $k$ particles of types (or species) $\{\lambda_1, \ldots,\lambda_k\}$. Our convention is that particles with larger labels are stronger. The set of configurations of the model consists of all possible arrangements of the $k$ particles amongst the $n$ sites -- each site may contain an arbitrary number of particles, and particles of the same type are indistinguishable.  We may identify configurations as multiset compositions of type $\lambda$ with $n$ parts; that is, sequences $w=(w_1, \dots, w_n)$ where for $1 \leq j \leq n$, $w_j$ is a (possibly empty) multiset, such that the union of all the parts $\biguplus_{j=1}^n w_j$ is equal to the multiset $\lambda$. We denote the set of configurations by $\TAZRP(\lambda,n)$. 

\begin{figure}[h]
\centering
\begin{tikzpicture}[
  ->,   
  thick,
  main node/.style={circle, white,fill=none, draw},
]
  \newcommand*{\MainNum}{5}
  \newcommand*{\MainRadius}{1.5cm} 
  \newcommand*{\MainStartAngle}{90}

  \path
    (0, 0) coordinate (M)
    \foreach \t [count=\i] in {{\ \ \textcolor{black}{$\emptyset$} \ \ \ }, {\textcolor{black}{3,1,1}}, {\ \ \textcolor{black}{$\emptyset$}\ \ }, {\ \textcolor{black}{4,2,2}\ }, {\textcolor{black}{3,2,1}}} {
      +({\i-1)*360/\MainNum + \MainStartAngle}:\MainRadius)
      node[main node, align=center] (p\i) {\t}
    }
  ;  

    \foreach \i in {1, ..., \MainNum} {
    \pgfextracty{\dimen0 }{\pgfpointanchor{p\i}{north}} 
    \pgfextracty{\dimen2 }{\pgfpointanchor{p\i}{center}}
    \dimen0=\dimexpr\dimen2 - \dimen0\relax 
    \ifdim\dimen0<0pt \dimen0 = -\dimen0 \fi
    \pgfmathparse{2*asin(\the\dimen0/\MainRadius/2)}
    \global\expandafter\let\csname p\i-angle\endcsname\pgfmathresult
  }

  \foreach \i [evaluate=\i as \nexti using {int(mod(\i, \MainNum)+1}]
  in {1, ..., \MainNum} {  
    \pgfmathsetmacro\StartAngle{   
      (\i-1)*360/\MainNum + \MainStartAngle
      + \csname p\i-angle\endcsname
    }
    \pgfmathsetmacro\EndAngle{
      (\nexti-1)*360/\MainNum + \MainStartAngle
      - \csname p\nexti-angle\endcsname
    }
    \ifdim\EndAngle pt < \StartAngle pt
      \pgfmathsetmacro\EndAngle{\EndAngle + 360}
    \fi
    \draw
      (M) ++(\EndAngle:\MainRadius)
      arc[start angle=\EndAngle, end angle=\StartAngle, radius=\MainRadius]
    ;
  }
   \node at (-3,0) {(a)};
\end{tikzpicture}
\qquad\qquad\qquad
\begin{tikzpicture}[
  ->,   
  thick,
  main node/.style={circle, white,fill=none, draw},
]
  \newcommand*{\MainNum}{5}
  \newcommand*{\MainRadius}{1.5cm} 
  \newcommand*{\MainStartAngle}{90}

  \path
    (0, 0) coordinate (M)
    \foreach \t [count=\i] in {{\ \ \textcolor{black}{$\emptyset$} \ \ \ }, {\textcolor{black}{3,1,1}}, {\ \ \textcolor{black}{$\emptyset$}\ \ }, {\ \textcolor{black}{4,2,2,2}\ }, {\textcolor{black}{3,1}}} {
      +({\i-1)*360/\MainNum + \MainStartAngle}:\MainRadius)
      node[main node, align=center] (p\i) {\t}
    }
  ;

  \foreach \i in {1, ..., \MainNum} {
    \pgfextracty{\dimen0 }{\pgfpointanchor{p\i}{north}} 
    \pgfextracty{\dimen2 }{\pgfpointanchor{p\i}{center}}
    \dimen0=\dimexpr\dimen2 - \dimen0\relax 
    \ifdim\dimen0<0pt \dimen0 = -\dimen0 \fi
    \pgfmathparse{2*asin(\the\dimen0/\MainRadius/2)}
    \global\expandafter\let\csname p\i-angle\endcsname\pgfmathresult
  }

  \foreach \i [evaluate=\i as \nexti using {int(mod(\i, \MainNum)+1}]
  in {1, ..., \MainNum} {  
    \pgfmathsetmacro\StartAngle{   
      (\i-1)*360/\MainNum + \MainStartAngle
      + \csname p\i-angle\endcsname
    }
    \pgfmathsetmacro\EndAngle{
      (\nexti-1)*360/\MainNum + \MainStartAngle
      - \csname p\nexti-angle\endcsname
    }
    \ifdim\EndAngle pt < \StartAngle pt
      \pgfmathsetmacro\EndAngle{\EndAngle + 360}
    \fi
    \draw
      (M) ++(\EndAngle:\MainRadius)
      arc[start angle=\EndAngle, end angle=\StartAngle, radius=\MainRadius]
    ;
  }
  \node at (-3,0) {(b)};
\end{tikzpicture}
\caption{(a) The configuration in this figure is $(\cdot|321|422|\cdot|311)$ (read clockwise starting from the topmost site), and is an element of $\TAZRP(\lambda,5)$ for $\lambda=(4,3,3,2,2,2,1,1,1)$. The arrows show the direction of hopping. (b) The configuration is $(\cdot|31|4222|\cdot|311)$ after a particle of type 2 hops from site 2 to site 3. This hop occurs with rate $x_2^{-1}t$.}\label{fig:tazrp}
\end{figure}

The system evolves as a continuous-time Markov chain with a global parameter $t$ and site-dependent parameters $x_1,\ldots,x_n$. Any transition of the system consists of a single particle jumping from site $j$ to site $j+1$, for some $1 \leq j\leq n$ (sites are considered cyclically mod $n$, so that a particle jumping out of site $n$ enters site $1$). See \cref{fig:tazrp} for an example.

The transition rates are defined as follows. For each $j\in\{1,\dots,n\}$ and each $a\geq 1$, a bell of level $a$ rings at site $j$ at rate
$x_j^{-1} t^{a-1}$. 
When such a bell rings: if site $j$ contains at least $a$ particles, 
then the $a$'th particle (when numbered from largest to smallest) jumps to site $j+1$. If $j$ contains fewer than $a$ particles, nothing changes. Expressed in another way, if the number of particles of type $r$ at site $j$ is $c_r$, and the number of particles of type 
larger than $r$ at site $j$ is $d_r$, then the total rate of a particle of type $r$ jumping from site $j$ to site $j+1$ is 
$x_j^{-1} t^{d_r}\sum_{i=0}^{c_r-1} t^{i}$.

\begin{example}
$\TAZRP((3,1,1),3)$ consists of the following
$18$ states: 
\begin{align*}
(311|\cdot|\cdot), (31|1|\cdot), (31|\cdot|1), (3|11|\cdot), (3|1|1), (3|\cdot|11), (11|3|\cdot), (11|\cdot|3), (1|31|\cdot),\\ 
(1|1|3), (1|3|1), (1|\cdot|31), (\cdot|311|\cdot), (\cdot|31|1), (\cdot|3|11), (\cdot|11|3), (\cdot|1|31), (\cdot|\cdot|311).
\end{align*}
Examples of transitions of the mTAZRP on $\TAZRP((3,1,1),3)$ are: 
\begin{itemize}[itemsep=0pt]
\item The jumps from $(\cdot|311|\cdot)$ are to $(\cdot|11|3)$ with rate $x_2^{-1}$, and
to $(\cdot|31|1)$ with rate $x_2^{-1}(t+t^2)$;
\item The jumps from $(\cdot|1|31)$ are to $(3| 1| 1)$ with rate $x_3^{-1}$, 
to $(1| 1|3)$ with rate $x_3^{-1}t$, and to $(\cdot|\cdot|311)$ with rate $x_2^{-1}$. 
\end{itemize}
\end{example}

A multi-species TAZRP with $r$ different species of particles
can be seen as a coupling of $r$ single-species TAZRPs: 
\begin{prop}
\label{prop:colouring}
Let $n$ be a positive integer, $\lambda = (\lambda_1,\dots,\lambda_k)$ be a partition, and $j \in [k-1]$ such that $\lambda_j \neq \lambda_{j+1}$. Then the multispecies TAZRP on $\TAZRP(\lambda,n)$ lumps to the 
single species TAZRP on $\TAZRP(\langle 1^j \rangle,n)$.
\end{prop}

\begin{proof}
Notice first that the jump rates of particles of species $a$
do not depend on the locations of species 
with lower labels. If we ignore all particles
with labels lower than some threshold, the remaining particles
still evolve as a multispecies TAZRP (with fewer classes). 

Secondly, relabelling all particles with labels in some interval $[a,b]$ as $a$ does not affect the total jump rate of that collection of particles -- this operation again results in an mTAZRP with fewer classes. 

Combining these operations, we may ignore all the particles
whose label is lower than $\lambda_j$, and relabel all particles
with labels in $[\lambda_1, \lambda_j]$ as a single species, to 
obtain a single-species TAZRP with $j$ particles. 
\end{proof}

\section{Preliminaries on fillings of diagrams}
\label{sec:Tab}

Let $\lambda=(\lambda_1,\ldots,\lambda_k)$ be a partition. The \emph{diagram of type $\lambda$}, which we denote $\dg(\lambda)$, consists of
the cells $\{(r,i), 1\leq i\leq k, 1\leq r\leq \lambda_i\}$, which we depict using
$k$ bottom-justified columns, where the $i$'th column from left to right has $\lambda_i$ boxes. See \cref{fig:readingorder} for an illustration of $\dg(3,2,1,1)$.
The cell $(r,i)$ corresponds to the cell in the $i$'th column in the $r$'th row of $\dg(\lambda)$, where rows are labeled from bottom to top. 
We warn the reader that this is the opposite convention of labelling points in the plane using cartesian coordinates, but is somewhat standard in the literature. We will use this convention throughout the article. 

A \emph{filling} of type $(\lambda,n)$ is a function $\sigma:\dg(\lambda) \rightarrow [n]$ defined on the cells of $\dg(\lambda)$ (where $[n]=\{1,2,\dots, n\}$).
We also refer to a diagram together with a filling of it as a \emph{tableau}. Let $\Tab(\lambda,n)$ be the set of fillings of type $(\lambda, n)$, and $\Tab(\lambda)$ the set of fillings $\sigma:\dg(\lambda) \rightarrow \mathbb{Z}^+$.

\begin{center}
\begin{figure}[h]
\[
\tableau{1\\3&2\\7&6&5&4}
\]
\caption{The diagram $\dg(3,2,1,1)$ with the cells labeled according to the reading order in \cref{def:readingorder}. Our convention is that the cell labelled $5$ has coordinates $(1,3)$.}
\label{fig:readingorder}
\end{figure}
\end{center}

For $\sigma$ in $\Tab(\lambda, n)$ (or in $\Tab(\lambda)$), and a cell $x=(r,i)$ of $\dg(\lambda)$, we write $\sigma(x)=\sigma(r,i)$ and call this the \emph{content} of the cell $x$ in the filling $\sigma$. For convenience, we also set the convention $\sigma(\lambda_i+1,i)=0$, where $(\lambda_i+1,i)$ is the nonexistent cell above column $i$. Define $\South(x)$ and $\North(x)$ to be the cells $(r-1,i)$ and $(r+1,i)$, directly below and above cell $x$ (respectively) in the same column; if $x$ is the bottom-most (\emph{resp.}~top-most) cell of its column, $\South(x)=\emptyset$ (\emph{resp.}~$\North(x)=\emptyset$). 

\begin{defn}
\label{def:readingorder}
Define the \emph{reading order} on the cells of a tableau to be along the rows from right to left,
taking the rows from top to bottom. 
\end{defn}

See \cref{fig:readingorder} for an illustration of the reading order.

\begin{defn}\label{def:quinv}
Given a diagram $\dg(\lambda)$, a \emph{triple} consists of either
\begin{itemize}
\item three cells $(r+1,i)$, $(r,i)$ and $(r,j)$ with $i<j$; or
\item two cells $(r,i)$ and $(r,j)$ with $i<j$ and  $\lambda_i = r$
(in which case the triple is called a \emph{degenerate triple}).
\end{itemize}
Let us write $a=\sigma(r+1 ,i)$, $b=\sigma(r,i)$, and $c=\sigma(r,j)$, 
for the contents of the cells of the triple, so that we can 
depict a triple along with its contents as
\begin{center}
\raisebox{-10pt}{\qtrip{$a$}{$b$}{$c$}}, \qquad or \qquad \raisebox{-10pt}{\begin{tikzpicture}[scale=0.5]
\node at (-0.5,-.5) {$0$};
 \cell{2}{0}{$b$} \cell{2}{2.7}{$c$}
\node at (1,-1.5) {$\cdots$};
\end{tikzpicture}} \quad (degenerate).
\end{center}

We say that a triple is a \emph{queue inversion triple}, or a \emph{quinv triple} for short, if its entries are oriented counterclockwise when the entries are read in increasing order, with ties being broken with respect to reading order. If the triple is degenerate with content $b,c$, it is a $\quinv$ triple if and only if $b<c$.
(This is equivalent to thinking of a degenerate triple as a regular triple with $a=0$).
\end{defn}

Accordingly, define $\cQ$ to be the set of ordered triples of contents such that 
\begin{equation}
\label{cQdef}
\cQ=\big\{(a,b,c)\in[n]^3:
a<b<c,
\text{ or } b<c<a,
\text{ or } c<b<a,
\text{ or } a=b\ne c \big\}.
\end{equation}
Then, the triple $((r+1,i),(r,i),(r,j))$, where $i<j$, forms a quinv triple in $\sigma$ if and only if\linebreak $(\sigma(r+1,i),\sigma(r,i),\sigma(r,j))\in \cQ$.

 \begin{figure}[ht!]
 \begin{tikzpicture}[scale=0.5]
  \node at (0,0) {\tableau{3\\1&1&2\\1&3&2&3&1}};
 \end{tikzpicture}
  \centering
  \caption{A tableau of type $\lambda=(3,3,2,2,2,1,1)$ and $n=3$. The weight of this filling is $x_1^4x_2^2x_3^3t^7$.} 
  \label{fig:polyqueue}
 \end{figure}

\begin{defn}\label{def:stats}
The \emph{weight} of a filling $\sigma$ is $x^{\sigma}t^{\quinv(\sigma)}$, where:
\begin{itemize}
\item $x^{\sigma}=\prod_{u\in\dg(\lambda)} x_{\sigma(u)}$ is the monomial corresponding to the content of $\sigma$,
\item $\quinv(\sigma)$ is the total number of $\quinv$ triples in $\sigma$.
 \end{itemize}
 \end{defn}
 
See \cref{fig:polyqueue} for an example of a tableau and its weight.

Originally, the primary motivation for studying fillings with the quinv statistic was to obtain a new tableaux formula for the modified Macdonald polynomials, which we showed in the prequel to this article \cite{AMM20}.
\begin{thm}[\cite{AMM20}]\label{thm:H}
Let $\lambda$ be a partition. The modified Macdonald polynomial can be written as
\[
\widetilde{H}_{\lambda}(X;q,t) = \sum_{\sigma\in\Tab(\lambda)} x^{\sigma}t^{\quinv(\sigma)}q^{\maj(\sigma)}.
\]
where $\maj(\sigma)$ is the major index statistic, defined in \cite{HHL05}.
\end{thm}

The main focus of this article will be to obtain formulas for 
the stationary probabilities and correlations of the mTAZRP on $\TAZRP(\lambda, n)$
using weights on the fillings in $\Tab(\lambda,n)$. 

We associate some set of fillings in $\Tab(\lambda,n)$ to each state of the TAZRP by mapping the bottom row of a filling in $\Tab(\lambda,n)$ to a state of a TAZRP of type $\lambda$ on $n$ sites. The correspondence is given as a function $\proj: \Tab(\lambda,n) \rightarrow \TAZRP(\lambda,n)$, as follows.

\begin{defn}\label{def:f}
Let $\sigma\in\Tab(\lambda,n)$. Then $\proj(\sigma)=(w_1,\ldots,w_n)$, where for $1\leq j \leq n$, 
\[
w_j=\{\lambda_i\in\lambda\ :\ \sigma(1,i)=j\}
\]
is the multiset of the heights of the columns of $T$ whose bottom-most entry is $j$. 
\end{defn}

\begin{example} We show all six tableaux corresponding to $w=(\cdot|21|1) \in\TAZRP((2,1,1),3)$:

\begin{center}
\begin{tikzpicture}[scale=0.5]
\cell00{$1$}
\cell10{$2$}\cell11{$3$}\cell12{$2$}

\cell05{$2$}
\cell15{$2$}\cell16{$3$}\cell17{$2$}

\cell0{10}{$3$}
\cell1{10}{$2$}\cell1{11}{$3$}\cell1{12}{$2$}

\cell0{15}{$1$}
\cell1{15}{$2$}\cell1{16}{$2$}\cell1{17}{$3$}

\cell0{20}{$2$}
\cell1{20}{$2$}\cell1{21}{$2$}\cell1{22}{$3$}

\cell0{25}{$3$}
\cell1{25}{$2$}\cell1{26}{$2$}\cell1{27}{$3$}
\end{tikzpicture}
\end{center}
\end{example}

One of the main results in this article, \cref{thm:TAZRP-stat-distn}, states that the stationary probability of a state $w\in\TAZRP(\lambda,n)$ can be computed as the sum over the weights of the pre-image of $w$ in $\Tab(\lambda,n)$ under the map $\proj$.

\subsection{The operators $\{\tau_j\}$}
\label{sec:tau}

Let $\sigma\in\Tab(\lambda,n)$ be a filling.
Suppose two columns $j$ and $j+1$ have equal height, i.e.\
$\lambda_j=\lambda_{j+1}=k$. We define an operator $\tau_j$ which
exchanges contents of certain cells between columns $j$ and $j+1$. 
Write $\sigma(r,j)=a_r$ and $\sigma(r,j+1)=b_r$ for $r=1,\dots, k$. 
Define $r_{\max}$ to be the largest $r\in\{2,\dots,k\}$ with the following property:
either $(a_r, a_{r-1}, b_{r-1})$ and $(b_r, a_{r-1}, b_{r-1})$ are both in $\cQ$,
or both are not in $\cQ$. (That is, we look for the largest value of $r$ 
such that exchanging the entries $a_r$ and $b_r$
makes no difference to whether $\{(r,j), (r-1,j), (r-1, j+1)\}$ is a quinv triple.)
If there is no such $r$, let $r_{\max}=1$. 

Now let the operator $\tau_j$ swap the entries between columns $j$ and $j+1$
in rows $i$ with $r_{\max}\leq i\leq k$; i.e.\ for those values of $i$,
\[
\tau_j(\sigma)(i,j)=\sigma(i,j+1)=b_i,
\,\, 
\tau_j(\sigma)(i,j+1)=\sigma(i,j)=a_i,
\]
while all other entries are the same in $\sigma$ and in $\tau_j(\sigma)$, as in the picture below, where we denote $\ell:=r_{\max}$.

\begin{equation*}
\renewcommand{\arraystretch}{1}  
\setlength{\arraycolsep}{1pt} 
\begin{array}{l|c|c|}
\multicolumn{1}{c}{}&
\multicolumn{2}{c}{\sigma}
\\
\cline{2-3} 
\text{row } k &  a_k & b_k\\
 \cline{2-3} 
\,\,\,\vdots & \vdots & \vdots \\
\cline{2-3}
\text{row } \ell & a_{\ell} & b_{\ell}\\
\cline{2-3} 
\text{row } \ell-1\,\,\,\,\,\, & a_{\ell-1} & b_{\ell-1}\\
 \cline{2-3}
\,\,\,\vdots & \vdots & \vdots \\
 \cline{2-3}
 \multicolumn{1}{c}{}&
\multicolumn{1}{c}{\scriptstyle j}&
\multicolumn{1}{c}{\scriptstyle j+1}
\end{array}
\,\,\xrightarrow{\makebox[1.5cm]{$\tau_j$}}\,\,
\begin{array}{|c|c|}
\multicolumn{2}{c}{\tau_j(\sigma)}
\\
\hline
b_k &  a_k\\
\hline
\vdots & \vdots \\
\hline
b_{\ell} & a_{\ell}\\
\hline
a_{\ell-1} & b_{\ell-1}\\
\hline
\vdots & \vdots \\
\hline
\multicolumn{1}{c}{\scriptstyle j}&
\multicolumn{1}{c}{\scriptstyle j+1}
\end{array}
\end{equation*}

Put another way: we swap the pair in columns $j,j+1$ in row $k$, and iteratively, if we have swapped the pair in columns $j,j+1$ in row $i$ and this made a difference to whether the triple $\{(i,j), (i-1, j), (i-1, j+1)\}$ is a quinv
triple, then we perform the swap in columns $j,j+1$ of row $i-1$ also. 

\begin{example}\label{ex:tau}
Suppose $\sigma$ has columns $j,j+1$ as shown below. Then $k=5$ and $r_{\max}=3$ since both $(2,3,4)\in\cQ$ and $(3,3,4)\in\cQ$. Thus applying the operator $\tau_j$ gives the following. The cells whose content was swapped are shown in grey.

\begin{equation*}
\renewcommand{\arraystretch}{1}  
\setlength{\arraycolsep}{1pt} 
\begin{array}{|c|c|}
\multicolumn{2}{c}{\sigma}
\\
\hline
3 &  4\\
\hline
2 & 3  \\
\hline
2 & 3 \\
\hline
3 & 4\\
\hline
1 & 3 \\
\hline
\multicolumn{1}{c}{\scriptstyle \phantom{+} j\phantom{1}}&
\multicolumn{1}{c}{\scriptstyle j+1}
\end{array}
\,\,\xrightarrow{\makebox[1.5cm]{$\tau_j$}}\,\,
\begin{array}{|c|c|}
\multicolumn{2}{c}{\tau_j(\sigma)}
\\
\hline
\cellcolor{gray!50}4 & \cellcolor{gray!50} 3\\
\hline
\cellcolor{gray!50}3 & \cellcolor{gray!50}2  \\
\hline
\cellcolor{gray!50}3 & \cellcolor{gray!50}2 \\
\hline
3 & 4\\
\hline
1 & 3 \\
\hline
\multicolumn{1}{c}{\scriptstyle \phantom{+} j\phantom{1}}&
\multicolumn{1}{c}{\scriptstyle j+1}
\end{array}
\end{equation*}

\end{example}

The key property of $\tau_j$ that will be important to us is that  
\[
\quinv(\tau_j(\sigma);L)=\begin{cases} \quinv(\tau_j(\sigma);L),& L\neq (\emptyset,(\lambda_j,j),(\lambda_j,j+1)),\\
1- \quinv(\tau_j(\sigma);L),& L= (\emptyset,(\lambda_j,j),(\lambda_j,j+1)).
\end{cases}
\]
In other words, $\quinv(\sigma;L)$ and $\quinv(\tau_j(\sigma);L)$ are equal unless $L$ is the unique degenerate triple in columns $j, j+1$. 
This gives the following lemma.

\begin{lemma}[{\cite[Lemma 7.6]{AMM20}}]\label{lem:tauquinv}
Let $\sigma\in\Tab(\lambda, n)$ and let $j$ be such that the columns $j,j+1$ are of equal height $k$. Let $x=(k,j)$ and $y=(k,j+1)$ be the cells at the tops of columns $j$ and $j+1$. Then
\begin{equation}
\quinv(\tau_j(\sigma))=\quinv(\sigma) + \begin{cases} 1,& \sigma(x)>\sigma(y),\\
-1,&\sigma(x)<\sigma(y),\\
0,&\sigma(x)=\sigma(y). \end{cases}
\end{equation}
\end{lemma}

\section{Dynamics on $\Tab(\lambda,n)$}\label{sec:Markov}

Fix a partition $\lambda$, and let $n\in\NN$. In this section, we introduce a continuous-time Markov chain 
on the set $\Tab(\lambda,n)$ of fillings of $\dg(\lambda)$,
which we will call the \emph{tableau Markov chain} (or tableau process). 
The tableau process will have two key properties:
(i) its projection under the map $\proj$ is the mTAZRP; 
(ii) the stationary probability of the filling $\sigma$ is proportional
to $\wt(\sigma)=x^{\sigma} t^{\quinv(\sigma)}$.

We will start in \cref{sec:markov distinct} with the case 
where $\lambda$ is a strict partition, i.e.\ has all its parts distinct. 
This corresponds to the property that all particles of the mTAZRP 
have distinct types. In that case the process is quite simple to 
describe (and does not require the use of the operators 
$\tau_j$ defined in the previous section);
the case of general $\lambda$ is given in 
\cref{sec:markovgeneral}. 

We begin by making some definitions to prepare for the description of
the Markov chain. 

\begin{defn} Let $u=(r,c)$ be a cell of $\dg(\lambda)$. 
The \emph{lower arm} of $u$, denoted $\arm(u)$, 
consists of all cells to its left in the same row plus all cells to its right in the row below:
\[
\arm(u) := \{(r,j) \in \dg(\lambda)\ :\ j<c\} \cup \{(r-1,j)\in\dg(\lambda)\ :\ j>c\}.
\]
Similarly, we define the \emph{upper arm} of $u$, denoted $\uarm(u)$, to consist of all cells to its left in the row above plus all cells to its right in the same row:
\[
\uarm(u) := \{(r+1,j) \in \dg(\lambda)\ :\ j<c\} \cup \{(r,j)\in\dg(\lambda)\ :\ j>c\}.
\]
See \cref{fig:lowerupperarm} for an illustration of the lower and upper arms. Note that what we call the lower arm is the statistic `arm' in \cite{HHL05}. Observe that $v \in \arm(u)$ if and only if $ u \in \uarm(v)$.
\end{defn}

\begin{figure}[h!]
\centering
 \begin{tikzpicture}[scale=0.5]
 \node at (0,0) {$\tableau{u&u&\ \\\ell&\ell&x&u&u\\\ &\ &\ &\ell&\ell&\ell&\ell}$};
 \end{tikzpicture}
\caption{For the cell $x=(2,3)$, the cells in $\arm(x)$ and $\uarm(x)$ are labeled `$\ell$' and `$u$', respectively.}
\label{fig:lowerupperarm}
\end{figure}

Recall that $\South(u) = (r-1,c)$ is the cell directly below $u$, if it exists, and
that $\North(u)=(r+1,c)$ is the cell directly above $u$, if it exists. We shall use the conventions that if $\South(u)$ (resp.~$\North(u)$) doesn't exist, then $\sigma(\South(u))=\infty$ (resp.~$\sigma(\North(u))=0$). Now we define
\[
\down(\sigma,u):=
\begin{cases} 
-\infty, & \sigma(\South(u))=\sigma(u)\\ 
\#\{y\in\arm(u):\sigma(y)=\sigma(u)\},&\mbox{otherwise}.
\end{cases}
\]
and its analog
\[
\up(\sigma,u):=
\begin{cases} 
-\infty, & \sigma(\North(u))=\sigma(u)\\ 
\#\{y\in\uarm(u):\sigma(y)=\sigma(u)\},&\mbox{otherwise}.
\end{cases}
\]

Let $\lambda$ be a partition, let $u=(r,c)\in\dg(\lambda)$ be a cell, 
and let $\sigma\in\Tab(\lambda, n)$ be a filling of $\lambda$.

We shall consider the maximal vertical chain of cells upwards from $u$
along which the contents increase by $1$ (cyclically mod $n$) at each step.
Denote the cell at the top of the chain by $H_u(\sigma)$. 
Then $H_u(\sigma)=(r+m,c)$, where $r\leq r+m\leq \lambda_c$ such that 
\begin{itemize}
\item[(i)] for all $i$ with $r\leq i < r+m$, we have 
$\sigma(\North(i,c))-\sigma(i,c)\equiv 1 \pmod n$;
\item[(ii)] either $r+m=\lambda_c$, or 
$\sigma(\North(H_u(\sigma)))-\sigma(H_u(\sigma)) \not\equiv 1\pmod n$.
\end{itemize}
Similarly, we shall consider the maximal chain of cell \emph{downwards} from $u$ along with the contents \emph{decrease} by 1 (again, cyclically mod $n$) at each step. Denote the cell at the bottom of the chain by $D_u(\sigma)$. Then if $D_u(\sigma)=(r-m,c)$, where $1\leq r-m\leq r$ such that
\begin{itemize}
\item[(i)] for all $i$ with $r-m< i \leq r$, we have 
$\sigma(i,c)-\sigma(\South(i,c))\equiv 1 \pmod n$;
\item[(ii)] either $r-m=1$, or 
$\sigma(D_u(\sigma))-\sigma(\South(D_u(\sigma))) \not\equiv 1\pmod n$.
\end{itemize}
\begin{defn}\label{defn:ringing-path}
Given $\lambda$ and a cell $u=(r,c)$, we define a \emph{ringing path map} 
$R_u : \Tab(\lambda,n) \to \Tab(\lambda,n)$ as follows.
If $\sigma\in \Tab(\lambda, n)$ with $\sigma(\South(u))=\sigma(u)$, 
then let $R_u(\sigma)=\sigma$. 
Otherwise, the map $R_u$ will
increment by $1$ all the contents in the maximal contiguous increasing 
chain above $u$. Let $H_u(\sigma)=(r+m,c)$. Then for $v\in\dg(\lambda)$,
\[
R_u(\sigma)(v) = \begin{cases}
\sigma(v)+1 \pmod n\qquad\qquad& v=(i,c)\ \mbox{for}\ r \leq i \leq r+m,\\
\sigma(v)& \mbox{otherwise}.
\end{cases}
\]
\end{defn}

See \cref{fig:Ru big} and \cref{eg:Ru map} for examples.

\begin{figure}[h!]
\begin{center}
\begin{tikzpicture}[scale=0.49]
\node at (0,0) {$\sigma=$};
\begin{scope}[shift={(1.25,0)}]
\draw[fill=gray!50,draw opacity=0] (.95,-1.3)--(1.825,-1.3)--(1.825,1.275)--(.95,1.275)--(.95,-1.3);
\end{scope}
\node at (3.5,0) {$\tableau{1&4&3\\2&1&4\\1&4&2&3&2\\1&3&2&4&1\\2&2&1&3&3}$};
\node at (12,0) {$R_{(2,2)}(\sigma)=$};
\begin{scope}[shift={(14.75,0)}]
\draw[fill=gray!50,draw opacity=0] (.95,-1.3)--(1.825,-1.3)--(1.825,1.275)--(.95,1.275)--(.95,-1.3);
\end{scope}
\node at (17,0) {$\tableau{1&4&3\\2&2&4\\1&1&2&3&2\\1&4&2&4&1\\2&2&1&3&3}$};
\end{tikzpicture}
\end{center}
\caption{The map $R_u$ is shown for the cell $u=(2,2)$ for the tableau $\sigma$ with $n=4$. The maximal vertical chain beginning at $u$ with contiguously increasing contents is shown in grey. The cell at the top of the chain is $H_u(\sigma)=(4,2)$, which is precisely the cell $y$ referred to in \cref{lem:g} (i).}\label{fig:Ru big}
\end{figure}

\begin{defn}\label{defn:ringing-path-down}
For a cell $u=(r,c)\in\dg(\lambda)$, we similarly define a \emph{reverse ringing path map} $B_u: \Tab(\lambda,n) \to \Tab(\lambda,n)$ as follows. If $\sigma(\North(u)) = \sigma(u)$, then $B_u(\sigma) = \sigma$. If not, the map $B_u$ will decrement by $1$ all the contents in the maximal contiguous decreasing chain downards from $u$. Let $D_u(\sigma)=(r-m,c)$. Then for $v\in \dg(\lambda)$,
\[
B_u(\sigma)(v) = \begin{cases}
\sigma(v)-1 \pmod n\qquad\qquad& v=(i,c)\ \mbox{for}\ r-m \leq i \leq r,\\
\sigma(v)& \mbox{otherwise}.
\end{cases}
\]
\end{defn}
We will show that the maps 
\[
R,B\colon \dg(\lambda)\times \Tab(\lambda,n)\rightarrow \dg(\lambda)\times \Tab(\lambda,n)
\]
defined by $R(u;\sigma)=(H_u(\sigma),R_u(\sigma))$ and $B(u;\sigma)=(D_u(\sigma),B_u(\sigma))$, are mutually inverse bijections. The following properties are straightforward to check from the definitions, from which it follows that $B(H_u(\sigma),R_u(\sigma))=R(D_u(\sigma),B_u(\sigma))=(u,\sigma)$.

\begin{lemma}\label{lem:g}
\begin{itemize}
\item[(i)]
Take $\xi\in\Tab(\lambda, n)$, and $u=(r,j)\in\dg(\lambda)$ such that $\xi(u)\neq\xi(\South(u))$; define $\sigma:=R_u(\xi)$. 
Then $y=H_u(\xi)$ is the unique cell such that  
$\xi=B_y(\sigma)$,
and has the property that $\sigma(y)\neq\sigma(\North(y))$. Moreover,
\begin{equation}\label{eq:sigma-xi-relation}
x^{\xi}=x^{\sigma}x_{\xi(u)}x^{-1}_{\sigma(y)}.
\end{equation}

\item[(ii)]
Conversely, take $\sigma\in\Tab(\lambda,n)$, and $y\in\dg(\lambda)$ 
such that $\sigma(y)\neq\sigma(\North(y))$; define $\xi:=B_y(\sigma)$. 
Then there is a unique $u$ such that $\sigma=R_u(\xi)$; moreover, $\xi(u)\neq\xi(\South(u))$ and $H_u(\xi)=y$. 
\end{itemize}
\end{lemma}


\begin{example}
\label{eg:Ru map}
Let $\sigma\in\Tab((3,2,1),3)$ be the tableau\begin{center}
\begin{tikzpicture}[scale=0.5]
\node at (0,0) {\tableau{1\\1&3\\3&2&3}};
\end{tikzpicture}
\end{center}
corresponding to the mTAZRP state $\proj(\sigma) = (\cdot | 2 | 31)$. 
Note that 
$R_{(3,1)}(\sigma) =B_{(2,1)}(\sigma)= \sigma$.
Below, we show the map $R$ for the cells $(1,1)$ and $(2,2)$, as well as the inverse map $B$ for the cells $(3,1)$ and $(2,2)$, with the corresponding maximal chain highlighted.
\begin{center}
\begin{tikzpicture}[scale=0.49]
\begin{scope}[shift={(-1.7,0)}]
\draw[fill=gray!50,draw opacity=0] (.9,-1.3)--(1.8,-1.3)--(1.8,-.4)--(.9,-.4)--(.9,-1.3);
\end{scope}
\begin{scope}[shift={(-1.7,0+.85)}]
\draw[fill=gray!50,draw opacity=0] (.9,-1.3)--(1.8,-1.3)--(1.8,-.4)--(.9,-.4)--(.9,-1.3);
\end{scope}
\begin{scope}[shift={(8.5-.85,0+.85)}]
\draw[fill=gray!50,draw opacity=0] (.9,-1.3)--(1.8,-1.3)--(1.8,-.4)--(.9,-.4)--(.9,-1.3);
\end{scope}
\begin{scope}[shift={(17-1.7,+1.7)}]
\draw[fill=gray!50,draw opacity=0] (.9,-1.3)--(1.8,-1.3)--(1.8,-.4)--(.9,-.4)--(.9,-1.3);
\end{scope}
\begin{scope}[shift={(26-.85,0+.85)}]
\draw[fill=gray!50,draw opacity=0] (.9,-1.3)--(1.8,-1.3)--(1.8,-.4)--(.9,-.4)--(.9,-1.3);
\end{scope}
\begin{scope}[shift={(26-.85,0)}]
\draw[fill=gray!50,draw opacity=0] (.9,-1.3)--(1.8,-1.3)--(1.8,-.4)--(.9,-.4)--(.9,-1.3);
\end{scope}
\node at (-3.5,0) {$R_{(1,1)}(\sigma) =$};
\node at (.5,0) {\tableau{1\\2&3\\1&2&3}};
\node at (-2,-2) {rate $x_3^{-1}$};
\node at (-2,-3.2) {state $(3|2|1)$};
\node at (5,0) {$R_{(2,2)}(\sigma) =$};
\node at (9,0) {\tableau{1\\1&1\\3&2&3}};
\node at (6.5,-2) {rate $x_2^{-1}$};
\node at (6.5,-3.2) {state $(\cdot|2|31)$};
\node at (13.5,0) {$B_{(3,1)}(\sigma) =$};
\node at (17.5,0) {\tableau{3\\1&3\\3&2&3}};
\node at (15,-2) {$u=(3,1)$};
\node at (15,-3) {state $(\cdot|2|31)$};
\node at (22.5,0) {$B_{(2,2)}(\sigma) =$};
\node at (26.5,0) {\tableau{1\\1&2\\3&1&3}};
\node at (24,-2) {$u=(1,2)$};
\node at (24,-3) {state $(2|\cdot|31)$};
\end{tikzpicture}
\end{center}
The inverses of the above maps are given by:
\[
B_{(2,1)}\circ R_{(1,1)}(\sigma)=B_{(2,2)}\circ R_{(2,2)}(\sigma)=R_{(3,1)}\circ B_{(3,1)}(\sigma)=R_{(2,2)}\circ B_{(2,2)}(\sigma)=\sigma.
\]
\end{example}

The next few results will be the main ingredients in our proofs. The proofs are combinatorial in nature but rather technical, so we postpone them to \cref{sec:proofs}.

 The first lemma concerns the fact that for a given filling $\sigma$ and for each fixed value $k$, the multiset of values of the $\down$ statistic is equal to the multiset of values of the $\up$ statistic, over all cells with content $k$ in $\sigma$. 
  
\begin{lemma}
\label{lem:updown}
For all $\sigma\in\Tab(\lambda)$ and all $k$, we have 
\begin{equation}\label{eq:D=U}
\sum_{\substack{u \in \dg(\lambda)\\\sigma(u)=k\neq\sigma(\North(u))}} t^{\up(\sigma,u)}=\sum_{\substack{u \in \dg(\lambda)\\\sigma(u)=k\neq\sigma(\South(u))}} t^{\down(\sigma,u)}.
\end{equation}
\end{lemma}

The second main ingredient in our proofs is a relation for the change in the number of $\quinv$ triples occurring under the maps $R_u$. Recall the \emph{Kronecker delta} $\delta_{i,j}$ which equals 1 if $i=j$ and $0$ otherwise.

\begin{lemma}
\label{lem:quinvdiff}
Fix a partition $\lambda$ and $n\in\NN$. For $\xi\in\Tab(\lambda,n)$, let $u=(r,j)\in\dg(\lambda)$ with $\xi(\South(u))\neq \xi(u)$, let $\sigma = R_u(\xi)$, and let $y=H_u(\xi)=(h_u(\xi),j)$. 
Then
\begin{multline}\label{eq:quinvdiff}
\quinv(\xi)-\quinv(\sigma) = \up(\sigma,y)-\down(\xi,u) \\
+ \delta_{\lambda_j,h_u(\xi)} \delta_{\xi(y),n}
\Big(\#\left\{j'<j:\lambda_{j'}=\lambda_j \right\}-\#\left\{j''>j:\lambda_{j''}=\lambda_j \right\}\Big).
\end{multline}
In particular, when $y$ is not part of a degenerate triple or when $\xi(y)\neq n$, we have 
\begin{equation}\label{eq:quinvdistinct}
\quinv(\xi)-\quinv(\sigma) =  \up(\sigma,y)-\down(\xi,u).
\end{equation}
Note that \eqref{eq:quinvdistinct} always holds when $\lambda$ is strict.
\end{lemma}

\subsection{Tableau process for the case of a strict partition $\lambda$}
\label{sec:markov distinct}

We now define the tableau Markov chain with state space $\Tab(\lambda,n)$ for the case where $\lambda$ is a strict partition (i.e.~has no repeated parts) in the following way. Let $\sigma\in\Tab(\lambda,n)$. 
We attach an exponential clock with rate $t^{\down(\sigma,u)}x_{\sigma(u)}^{-1}$ to every cell $u\in\dg(\lambda)$ such that $\sigma(\South(u))\neq \sigma(u)$. When this clock rings, we make a transition to $R_u(\sigma)$.
We say that such a transition is \emph{triggered by the cell $u$}. 
For $\sigma,\xi\in\Tab(\lambda,n)$, denote by $\rate(\sigma,\xi)$ the rate of the transition from $\sigma$ to $\xi$.

We will later show in a more general setting that this process is 
irreducible (see \cref{prop:tab process irred}).
Assuming that for now, we have the following.

\begin{thm}
\label{thm:maintheorem0}
The stationary distribution of the tableau Markov chain defined above is proportional to 
\[
\wt(\sigma) = t^{\quinv(\sigma)}x^{\sigma}.
\]
\end{thm}

\begin{proof}[Proof of \cref{thm:maintheorem0}]

Denote all states reached from $\sigma$ in one step by $\Out(\sigma)$, and all states that reach $\sigma$ in one step by $\In(\sigma)$:
\begin{align*}
\Out(\sigma) &= \{R_u(\sigma)\ :\ u\in\dg(\lambda),\ \sigma(\South(u))\neq\sigma(u)\},\\
\In(\sigma)& = \{\xi\in\Tab(\lambda,n)\ :\ \sigma\in \Out(\xi)\}.
\end{align*}
From \cref{lem:g}(ii) we have that
$\In(\sigma)=\{B_y(\sigma)\ :\ y\in\dg(\lambda),\ \sigma(\North(y))\neq \sigma(y)\}$. 
Then it suffices to show
\begin{equation}\label{eq:balance}
\wt(\sigma) \sum_{\nu\in \Out(\sigma)} \rate(\sigma,\nu) = \sum_{\xi\in \In(\sigma)} \wt(\xi)\rate(\xi,\sigma),
\end{equation}
which we can expand as
\begin{equation}\label{eq:balance2}
x^{\sigma}t^{\quinv(\sigma)} \sum_{\substack{u\in\dg(\lambda)\\\sigma(\South(u))\neq\sigma(u)}} x^{-1}_{\sigma(u)}t^{\down(\sigma,u)} = \sum_{\substack{y\in\dg(\lambda)\\\sigma(\North(y))\neq\sigma(y)}} \wt(B_y(\sigma))\rate(B_y(\sigma),\sigma).
\end{equation}

To treat the right-hand side of \eqref{eq:balance2}, 
let $\xi=B_y(\sigma)$ for some $y\in\dg(\lambda)$ such that 
$\sigma(\North(y))\neq\sigma(y)$, 
and let $u\in\dg(\lambda)$ be such that $R_u(\xi)=\sigma$. 
From \eqref{eq:sigma-xi-relation}, we have
$x^{\xi}=x^{\sigma}x^{-1}_{\sigma(y)}x_{\xi(u)}$. By \cref{lem:quinvdiff}, we 
also have 
\begin{equation}
\wt(\xi)\rate(\xi,\sigma) = x^{\xi}t^{\quinv(\xi)}\cdot x^{-1}_{\xi(u)}t^{\down(\xi,u)} =x^{\sigma}x^{-1}_{\sigma(y)}t^{\quinv(\sigma)+\up(\sigma,y)}.\label{eq:wtxi}
\end{equation}
Refining over the coefficients of $x$-monomials and plugging in
\eqref{eq:wtxi},
we get that  \eqref{eq:balance2}  is equivalent to 
\begin{align}
x^{\sigma}t^{\quinv(\sigma)} \sum_{k=1}^n x^{-1}_{k} \sum_{\substack{u\in\dg(\lambda)\\\sigma(\South(u))\neq\sigma(u)=k}} t^{\down(\sigma,u)} &= \sum_{k=1}^n \sum_{\substack{y\in\dg(\lambda)\\\sigma(\North(y))\neq\sigma(y)=k}} x^{\sigma}x^{-1}_{k}t^{\quinv(\sigma)+\up(\sigma,y)}\nonumber\\
&=x^{\sigma}t^{\quinv(\sigma)}\sum_{k=1}^n x^{-1}_{k} \sum_{\substack{y\in\dg(\lambda)\\\sigma(\North(y))\neq\sigma(y)=k}} t^{\up(\sigma,y)}.\label{eq:balancewt}
\end{align}
By \cref{lem:updown}, for each $k$ we have
\begin{equation*}
\sum_{\substack{u\in \dg(\lambda)\\ \sigma(\South(u))\neq\sigma(u)=k}} t^{\down(\sigma,u)} = \sum_{\substack{y\in \dg(\lambda)\\ \sigma(\North(y))\neq\sigma(y)=k}} t^{\up(\sigma,y)},
\end{equation*}
from which we obtain \eqref{eq:balancewt} as required.
\end{proof}

\subsection{Tableau process for general $\lambda$}\label{sec:markovgeneral}

Fix a partition $\lambda$ and a positive integer $n$. We generalize the  tableau Markov chain defined in \cref{sec:markov distinct} to the case when $\lambda$ is a general partition. 

If $\lambda$ has repeated parts, some degenerate quinv triples in a filling of $\dg(\lambda)$ may be lost or created after a ringing path transition, which proves to be a significant hurdle in our definition of the tableau process, as the necessary \eqref{eq:wtxi} would no longer hold. 
The contribution of degenerate triples is reflected in the last term of \eqref{eq:quinvdiff}.  Our strategy for treating this is to apply entry-swapping operators $\{\tau_j\}$ to the tableaux (defined in \cref{sec:tau}) in order to 
cancel the change in quinv arising from these degenerate triples. 

To focus our attention on the degenerate triples of $\dg(\lambda)$ within a given row $r$, we define the \emph{denerate segment} in row $r$ to be the set of cells in that row that have no cells above them; namely, they are the cells $(\lambda_j,j)$ where $\lambda_j=r$ (if any exist). In the rest of this article, we will be only be making use of the degenerate segment containing a certain designated cell that is determined by the filling; thus, we will be referring to it simply as the degenerate segment of the filling without explicitly specifying its row.

Let $\xi\in\Tab(\lambda,n)$, and let $u\in\dg(\lambda)$ be a cell that triggers the jump from $\xi$ to $\sigma=R_u(\xi)$. Let $y$ be the cell defined by $\xi=B_y(\sigma)$. If $\xi(y) \neq n$ or if $y$ is not part of a degenerate triple, the transition from $\xi$ triggered by $u$ is to $\sigma=R_u(\xi)$, as in the previous section. 

If $\xi(y)=n$, then $\sigma(y)=1$. If $y=(\lambda_j,j)$ is part of a degenerate triple, we must deal with the change in $\quinv$ that arises from the degenerate triples that involve $y$. In particular, if $y'=(\lambda_j,j')$ for $j'<j$ where $\lambda_{j'}=\lambda_j$, then
we have
\[
\quinv(\xi; (\emptyset,y',y))=\begin{cases}
1&\text{if}\ \xi(y')\neq n,\\
0&\text{if}\ \xi(y')=n,
\end{cases}\quad\text{whereas}\quad
\quinv(\sigma; (\emptyset,y',y))=0.
\] 
Similarly, if $y''=(\lambda_j,j'')$ for some $j''>j$ (so that $\lambda_{j''}=\lambda_j$), 
we have
\[
\quinv(\xi; (\emptyset,y,y''))=0,\quad\text{whereas}\quad
\quinv(\sigma; (\emptyset,y,y''))=\begin{cases}
1&\text{if}\ \xi(y'')\neq 1,\\
0&\text{if}\ \xi(y'')=1.
\end{cases}
\] 

We will correct for this by reflecting the content of the cell $y$ to the other side of its degenerate segment; namely, if $y$ is the $j$'th cell in a degenerate
segment of length $k$, the content of $y$ in $\sigma$ will be sent to the $k-j+1$'st cell in the degenerate segment. In order to not alter the number of non-degenerate $\quinv$ triples in $\sigma$, we will ``move'' the content of cell $y$ by applying the operators $\{\tau_i\}$ in a particular order. 

\begin{defn}
\label{def:tauy}
Let $\xi\in\Tab(\lambda,n)$, $u\in\dg(\lambda)$ with $\xi(\South(u))\neq \xi(u)$, and $\sigma=R_u(\xi)$; let $y=(r,v)$ be such that $\xi=B_y(\sigma)$. If $\xi(y)\neq n$ or $r<\lambda_v$, set $v'=v$. Otherwise, if $r=\lambda_v$, then $y$ is contained in a degenerate segment; suppose this degenerate segment consists of the cells $\{(r,s),(r,s+1),\ldots,(r,s+k)\}$ (accordingly, $s \leq v \leq s+k$). Then we set
\[
v' = 2s+k-v,
\]
as in the schematic diagram in \cref{fig:taushift}.
We shall send the content of the cell $y=(r,v)$ to the cell $y'=(r,v')$ by applying the sequence of operators $\tau_y$ defined as:
\[
\tau_y = \begin{cases} \tau_{v'-1}\circ\cdots\circ\tau_{v+1}\circ\tau_{v},& v<v',\\
id,&v=v'\\
\tau_{v'}\circ\cdots\circ\tau_{v-2}\circ\tau_{v-1},& v>v'.
\end{cases}
\]
Then define $R'_u(\xi)=\tau_y\circ R_u(\xi)$.
\end{defn}

\begin{figure}[!h]
\centering
\scalebox{1.4}{
\begin{tikzpicture}[scale=0.5]
\draw (-4.5-.6,-2)--(4,-2)--(4,0)--(3.5,0)--(3.5,.4)--(3,.4)--(3,.8)--(-2-.4,.8)--(-2-.4,1.2)--(-2.5-.6,1.2)--(-2.5-.6,1.6)--(-3-.6,1.6)--(-3-.6,2)--(-4.5-.6,2)--(-4.5-.6,-2);
\node at (-6-1.5,-.5) { \scriptsize $\dg(\lambda)=$};
\node at (-5.8,0.6) {\scriptsize $r$};
\node at (-2.25,-3) {\scriptsize $s$};
\node at (-.25,-3) {\scriptsize $v$};
\node at (0.75,-2.9) {\scriptsize $v'$};
\node at (2.75,-3) {\scriptsize $s+k$};

\draw[dashed,thick,draw=gray!60] (2.75,-2.5)--(2.75,.3);
\draw[dashed,thick,draw=gray!60] (.75,-2.5)--(.75,.3);
\draw[dashed,thick,draw=gray!60] (-2.25,-2.5)--(-2.25,.3);
\draw[dashed,thick,draw=gray!60] (-.25,-2.5)--(-.25,.3);
\draw[dashed,thick,draw=gray!60] (-5.5,.6)--(3,.6);
\draw[fill=white] (2.6,.8)--(2.6,.4)--(3,.4)--(3,.8)--(2.6,.8);
\draw[draw,fill=gray!50] (2.6-2,.8)--(2.6-2,.4)--(3-2,.4)--(3-2,.8)--(2.6-2,.8);
\draw[draw,fill=gray!50] (2.6-3,.8)--(2.6-3,.4)--(3-3,.4)--(3-3,.8)--(2.6-3,.8);
\draw[fill=white] (2.6-5,.8)--(2.6-5,.4)--(3-5,.4)--(3-5,.8)--(2.6-5,.8);
\draw[draw,fill=gray!50] (2.6-3,.8-1.5)--(2.6-3,.4-1.5)--(3-3,.4-1.5)--(3-3,.8-1.5)--(2.6-3,.8-1.5);
\draw[draw,fill=white] (2.6-3,.8-1.5)--(2.6-3,.4)--(3-3,.4)--(3-3,.8-1.5)--(2.6-3,.8-1.5);
\node at (-.65,.2) {\tiny $y$};
\node at (1.2,.2) {\tiny $y'$};
\node at (-.65,-1.2) {\tiny $u$};
\draw [decorate,
    decoration = {calligraphic brace,amplitude=5pt}] (1.1,1) --  (3,1);
    \draw [decorate,
    decoration = {calligraphic brace,amplitude=5pt}] (-2.4,1) --  (-0.5,1);
\node at (2.1,1.7) {\tiny $v-s$};
\node at (-1.4,1.7) {\tiny $v-s$};

\end{tikzpicture}
}
\caption{Let $\xi,u,\sigma=R_u(\xi)$, $y=(r,v)$, $s$, and $k$ be defined as above with $\xi(y)=n$ and thus $\sigma(y)=1$. We show schematically the cell $y'=(r,v')$ in relation to $y$ in $\dg(\lambda)$, in the case where $v<v'$ (left).}\label{fig:taushift}
\end{figure}

We now define the continuous-time Markov process on $\Tab(\lambda,n)$ for general $\lambda$ as follows. Let $\sigma\in\Tab(\lambda,n)$. 
We attach an exponential clock with rate $t^{\down(\sigma,u)}x_{\sigma(u)}^{-1}$ to every cell $u\in\dg(\lambda)$ such that $\sigma(\South(u))\neq \sigma(u)$. When this clock rings, we make a transition to $R'_u(\sigma)$.

See \cref{eg:R' transition} for an example of a new transition.

\begin{example}
\label{eg:R' transition}
Let $\lambda=(3,2,2,2,2,2,1)$ and $n=3$. Suppose 
\[
\xi=\raisebox{-12pt}{\begin{tikzpicture}[scale=0.5]
\node at (0.05,1.25) {\tableau{\\&&&&&\\&&\darkgreysq&&&&}};
\node at (0.05,1.25) {\tableau{2\\3&1&3&1&2&3\\3&1&2&2&1&1&2}};\end{tikzpicture}},
\] and let $u=(1,3)$ (shaded grey). Then 
\[
\sigma=R_u(\xi)=\raisebox{-12pt}{\begin{tikzpicture}[scale=0.5]
\node at (0.05,1.25) {\tableau{\\&&\darkgreysq&&\darkgreysq&\\&&&&&&}};
\node at (0.05,1.25) {\tableau{2\\3&1&\pmb{\bm{1}}&1&2&3\\2&1&\pmb{\bm{3}}&2&1&1&2}};\end{tikzpicture}},
\]
where the cells $y$ and $y'$ are shaded grey and the entries in the ringing path are bolded.
Since the cell $y=(2,3)$ is contained in a degenerate segment and $\xi(y)=n$ so that $\sigma(y)=1$, we need to apply $\tau_y$. Using the notation above, $s=2$, $k=4$, $v=3$, and so $v'=2s+k-v=5$ and $y'=(2,5)$. Thus $\tau_y=\tau_4\tau_3$. We then get
\[
\sigma'=R'_u(\xi)=\tau_4\tau_3 \left(\raisebox{12pt}{\tableau{2\\3&1&\pmb{\bm{1}}&1&2&3\\2&1&3&2&1&1&2}} \right) = \tau_4\left(\raisebox{12pt}{\tableau{2\\3&1&1&\pmb{\bm{1}}&2&3\\2&1&3&2&1&1&2}} \right) = \raisebox{12pt}{\tableau{2\\3&1&1&2&\pmb{\bm{1}}&3\\2&1&3&1&2&1&2}}.
\]
(The content of cell $y$ in $\sigma$ which is being moved by $\tau_y$ is indicated in bold.) Computing $\quinv(\xi)=12$, $\quinv(\sigma')=13$, $\down(\xi,u)=1$, and $\up(\sigma',y')=0$, we verify that $\quinv(\xi)-\quinv(\sigma') = \up(\sigma',y')-\down(\xi,u)$. 
\end{example}

\begin{prop}
\label{prop:tab process irred}
The tableau Markov chain on $\Tab(\lambda,n)$ is irreducible.
\end{prop}

\begin{proof}
When $t=0$, this proposition is proved in \cref{lem:irreducible0}. 
Here we cover the case $t>0$. Call $\tau_0\in\Tab(\lambda,n)$ the filling consisting of all $1$'s. We first show that for any $\sigma\in\Tab(\lambda,n)$, there exists some sequence of transitions to reach the state $\tau_0$. We do this by induction on the lowest row of $\sigma$ that contains an entry greater than 1; if $\sigma$ has no such row, $\sigma=\tau_0$ and we are done. Otherwise, let $r$ be the lowest such row, so that rows $1,\ldots,r-1$ of $\sigma$ contain only 1's. Setting $\sigma_0=\sigma$, iteratively perform transitions $R_{u_j}$ to obtain $\sigma_j=R_{u_j}(\sigma_{j-1})$, with $u_j$ being any cell (for example, we can choose the rightmost one) in row $r$ such that $\sigma_{j-1}(u_j)>1$. Since $t>0$, all such transitions have strictly positive rate. Eventually, this process will necessarily terminate at some $\sigma_k$ which has its rows $1,\ldots,r$ containing only 1's. Repeating this process for the remaining rows, from $r+1$ to $\lambda_1$, we necessarily arrive at $\tau_0$.

For example, let $\lambda=(3,2,2)$ and $n=3$. The sequence of transitions is (with incremented entries in bold):
\[
\sigma=\tableau{1\\3&1&3\\3&2&1} \rightarrow \tableau{1\\3&1&3\\3&\pmb{\bm{3}}&1}\rightarrow \tableau{1\\3&\pmb{\bm{2}}&3\\3&\pmb{\bm{1}}&1} \rightarrow \tableau{1\\3&2&3\\ \pmb{\bm{1}}&1&1}\rightarrow \tableau{1\\3&\pmb{\bm{1}}&2\\1&1&1} \rightarrow \tableau{1\\3&1&\pmb{\bm{3}}\\1&1&1} \rightarrow \tableau{1\\3&\pmb{\bm{1}}&1\\1&1&1} \rightarrow \tableau{1\\\pmb{\bm{1}}&1&1\\1&1&1} =\tau_0
\]

Now we show that any filling $\sigma'$ can be reached from $\tau_0$ by some sequence of transitions. We do this in the same manner, by taking ringing path transitions for each cell from the bottom row upwards until the desired content is achieved. Notice that every such ringing path transition will affect only the cell $u$ that triggers said transition, since at each step of the process, if $u$ is selected to trigger a transition, $\North(u)=1\neq u$ (as we are working from bottom to top). Moreover, by construction there will not be any transitions that result in an $n$ becoming a 1, and so no $\tau_j$'s need to be applied to the filling, and all cells remain in their original place.

For example, we show the sequence of transitions given by our algorithm from $\tau_0$ to the $\sigma$ above (with incremented entries in bold):
\[
\tau_0=\tableau{1\\1&1&1\\1&1&1} \rightarrow \tableau{1\\1&1&1\\1&\pmb{\bm{2}}&1} \rightarrow \tableau{1\\1&1&1\\\pmb{\bm{2}}&2&1} \rightarrow \tableau{1\\1&1&1\\\pmb{\bm{3}}&2&1} \rightarrow \tableau{1\\1&1&\pmb{\bm{2}}\\3&2&1} \rightarrow \tableau{1\\1&1&\pmb{\bm{3}}\\3&2&1} \rightarrow \tableau{1\\\pmb{\bm{2}}&1&3\\3&2&1} \rightarrow \tableau{1\\\pmb{\bm{3}}&1&3\\3&2&1} =\sigma
\]
This completes the proof.
\end{proof}

The main result of this section is the following generalization of \cref{thm:maintheorem0}.

\begin{thm}
\label{thm:maintheorem}
Consider the tableau Markov process described after \cref{def:tauy}, where for a state $\sigma\in\Tab(\lambda,n)$, each cell $u\in\dg(\lambda)$ is equipped with an exponential clock with rate $x_{\sigma(u)}^{-1}t^{\down(\sigma,u)}$ that triggers the transition to $R'_u(\sigma)$. The stationary distribution of this process is proportional to 
\[\wt(\sigma) = t^{\quinv(\sigma)}x^{\sigma}.\]
\end{thm}

We now give a generalization of \eqref{eq:quinvdistinct}.

\begin{lemma}\label{lem:repeatedquinvdiff}
Let $\sigma = R_u(\xi)$ for $u=(r',v)$ such that $\xi(u)\neq\xi(\South(u))$, and let $y=H_u(\xi)$ be the unique cell such that $\xi=B_y(\sigma)$. Let $\sigma'=R'_u(\xi)=\tau_y (\sigma)$ where $\tau_y$ (given in \cref{def:tauy}) sends the content of the cell $y$ to the cell $y'$, as in \cref{fig:taushift} (the cells $y'$ and $y$ are marked with a square above). 

Then
\begin{equation}\label{eq:repeatedbalance}
\quinv(\xi)-\quinv(\sigma') = \up(\sigma',y') - \down(\xi,u).
\end{equation}
\end{lemma}

\begin{proof}
Recall the statement of \cref{lem:quinvdiff}, which we rewrite here for convenience:
\begin{equation}\label{eq:quinvdiff2}
\begin{split}
\quinv(\xi)-\quinv(\sigma) =&
\up(\sigma,y)-\down(\xi,u)  \\
&+ \delta_{\lambda_v,h_u(\xi)} \delta_{\xi(y),n}
\Big(\#\left\{j'<v:\lambda_{j'}=\lambda_v \right\}-\#\left\{j''>v:\lambda_{j''}=\lambda_v \right\}\Big)
\end{split}
\end{equation}
where $(h_u(\xi),j):=H_u(\xi)$. First, if $h_u(\xi)<\lambda_v$ or $\xi(y)\neq n$, then $y'=y$ and $\sigma'=\sigma$, and we automatically get \eqref{eq:repeatedbalance} from the above, since the third term vanishes. Thus let us assume $h_u(\xi)=\lambda_v$ and that $\xi(y)=n$.

Next we consider $\quinv(\sigma)-\quinv(\sigma')$ when $y'\neq y$. Suppose the degenerate segment containing $y=(\lambda_v,v)$ consists of the cells $\{(\lambda_v,s),(\lambda_v,s+1),\ldots,(\lambda_v,s+k)\}$; accordingly $s \leq v \leq s+k$. Then $v' = 2s+k-v$ and $y'=(\lambda_v,v')$ is the cell to which the content $\sigma(y)=1$ is sent. Observe that this construction yields $s \leq v'\leq s+k$ and $v-s = (s+k)-v'$.

Suppose $v'>v$. Then the sequence of operators in $\tau_y$ is $\tau_v,\tau_{v+1},\ldots,\tau_{v'-1}$, in that order. Each swap moves the 1 to the adjacent column to its right. From \cref{lem:tauquinv}, for each $0\leq \ell\leq v'-v-1$, 
\[
\quinv(\tau_{v+\ell}\circ\tau_{v+\ell-1}\circ\cdots\circ\tau_{v}(\sigma)) - \quinv(\tau_{v+\ell-1}\circ\cdots\circ\tau_{v}(\sigma) ) = \delta_{\sigma(\lambda_v,v+\ell+1),1}-1
\]
and so 
\[
\quinv(\sigma')-\quinv(\sigma)=\quinv(\tau_y(\sigma))-\quinv(\sigma) = v-v' +\#\{v<\ell\leq v':\sigma(\lambda_v,\ell)=1\}.
\]
Similarly, if $v'<v$, we have
\[
\quinv(\sigma')-\quinv(\sigma)= v-v' -\#\{v'\leq \ell< v:\sigma(\lambda_v,\ell)=1\}.
\]
We also observe that 
\[
\up(\sigma',y')-\up(\sigma,y)=\begin{cases} -\#\{v<\ell\leq v':\sigma(\lambda_v,\ell)=1\}& \text{if}\ v'>v,\\
\#\{v<\ell\leq v':\sigma(\lambda_v,\ell)=1\}& \text{if}\ v'<v.
\end{cases}
\]
We combine the above to get
\[
(\up(\sigma',y')-\up(\sigma,y))+(\quinv(\sigma')-\quinv(\sigma)) = v-v' = 2v-2s-k.
\]
We now examine the third term of \eqref{eq:quinvdiff2}:
\[
\#\left\{j'<v:\lambda_{j'}=\lambda_v \right\}-\#\left\{j''>v:\lambda_{j''}=\lambda_v \right\} = (v-s)-(s+k-v) = 2v-2s-k,
\]
to obtain, 
\begin{align*}
\quinv(\xi)-\quinv(\sigma') =& (\quinv(\xi)-\quinv(\sigma)) + (\quinv(\sigma)-\quinv(\sigma')) \\
=& \big(\up(\sigma,y)-\down(\xi,u) +(2v-2s-k)\big) \\
&+ \big(\up(\sigma',y')-\up(\sigma,y)+(2s+k-2v)\big)\\
=& \up(\sigma',y')-\down(\xi,u),
\end{align*}
as desired, completing the proof.
\end{proof}

\begin{proof}[Proof of \cref{thm:maintheorem}]
The proof is identical to that of \cref{thm:maintheorem0}, except that \cref{lem:repeatedquinvdiff} replaces \cref{lem:quinvdiff}.
\end{proof}

\subsection{Projection from $\Tab(\lambda,n)$ onto $\TAZRP(\lambda,n)$}

In this section, we prove that the general Markov process on tableaux  described in \cref{sec:markovgeneral} projects to the TAZRP Markov chain. 
Recall the projection map $\proj$ from $\Tab(\lambda,n)$ to $\TAZRP(\lambda,n)$ given in \cref{def:f}.
The key idea is the following lemma.

\begin{lemma}
\label{lem:lumping}
Let $\sigma\in\Tab(\lambda,n)$.
\begin{itemize}
\item[(i)] For any cell $u$ in the first row of $\sigma$, 
there exists a transition from the TAZRP state $\proj(\sigma)\in\TAZRP(\lambda,n)$ to the TAZRP state $\proj(R'_u(\sigma))$.

\item[(ii)] Let $w=\proj(\sigma)\in\TAZRP(\lambda,n)$ be the corresponding TAZRP state of $\sigma$. Let $1 \leq j \leq n$ be a site, $1 \leq r \leq \lambda_1$ be a particle type such that $w_j = r$. 
Let $w'$ be the state in $\TAZRP(\lambda,n)$ obtained by moving particle $r$ one site to the right. Then 
\[
\sum_{\substack{u\in \dg(\lambda)\\ \proj(R'_u(\sigma))=w'}} \rate(\sigma,R'_u(\sigma)) = \rate(w,w').
\]
\end{itemize}
\end{lemma}

\begin{proof}
We first prove (i). If $u$ is not in the first row, $\proj(\sigma) = \proj(R'_u(\sigma))$, so the claim is trivial. If $u=(1,k)$ for some column $k$, one can check that $\proj(R'_u(\sigma))$ is equal to the state obtained from the jump from the state $\proj(\sigma)$ of a particle of type $\lambda_k$ from site $\sigma(u)$, as desired.

Now we prove (ii). Let $\{u_1,\ldots,u_k\}$ be the set of columns such that $\sigma(1,u_i)=j$ and $\lambda_{u_i} = r$ for $1 \leq i \leq k$. 
Then $\proj(R'_u(\sigma)) = w'$ if and only if $u = (1,u_i)$ for $1 \leq i \leq k$. Thus we wish to show that 
\[
\sum_{i =1}^k x_j^{-1} t^{\down(1,u_i)} = x_j^{-1} t^{d_{j,r}}\sum_{i=0}^{c_{j,r}-1} t^i,
\]
where $d_{j,r}$ is the number of particles of type larger than $r$ at site $j$ in $w$, and $c_{j,r}$ is the number of particles of type equal to $r$ at site $j$ in $w$. First we note that $k = c_{j,r}$ by definition. Now, $\down(1,u_i)$ consists of all cells with content $j$ that are to the left of $u_i$ in row 1, which is precisely the quantity $d_{j,r}+(i-1)$. This proves the above equality and completes the proof.
\end{proof}
 
 Thus we obtain the first main result of this article, \cref{thm:TAZRP-stat-distn}.

\section{Special cases of the mTAZRP}
\label{sec:special}

\subsection{The single-species TAZRP}
\label{sec:single}

As an example, we examine the case of the single-species TAZRP of size $n$, when $\lambda=\langle 1^m \rangle$. 
In general we have thought about TAZRP configurations as
multi-set compositions,  but in this more simple case we can 
regard a configuration in $\TAZRP(\langle 1^m \rangle,n)$ as a weak composition of $m$ of length $n$, say as  $\eta = (\eta_1,\dots,\eta_n)$,  where $\eta_i$ is the number of particles (of type $1$) at site $i$.

The fillings in $\Tab(\langle 1^m \rangle,n)$ are single row tableaux with integers in $[n]$ and
can be thought of as words of length $m$ in the alphabet $[n]$, which we denote by $W(m,n)$. 

Recalling the projection $\proj$ from \cref{def:f} that maps from $\Tab(\lambda,n)$ to states of the TAZRP of type $\lambda$ on $n$ sites, we 
can rewrite it for the single species TAZRP
as follows. 
Let $\sigma \in W(m,n)$ be a word. 
Then $\proj(\sigma)$ is the state $\eta=(\eta_1,\ldots,\eta_n)$ where 
$\eta_i=\#\{ j\,:\,\sigma_j=i\}$.
In other words, for a given configuration $\eta \in \TAZRP(\langle 1^m \rangle,n)$, the set $\{\sigma \in W(m,n)\,:\,\proj(\sigma)=\eta \}$ is precisely the set of permutations of the word $1^{\eta_1}2^{\eta_2}\ldots n^{\eta_n}$.

We define the weight of a word $\sigma$ to be $x^\sigma t^{\inv(\sigma)}$ where $\inv(\sigma) = \#\{i<j\,:\,\sigma_i < \sigma_j\}$ is the number of inversions, with the lowest and highest weights being words sorted in decreasing and increasing order, respectively. As usual, $x^\sigma=\prod_i x_{\sigma_i}$ is the monomial representing the content.

\begin{example}
The set of words corresponding to the state 
$(11|\cdot|11) \equiv (2,0,2) \in\TAZRP(\langle 1^4 \rangle, \allowbreak 3)$ is $\{(3,3,1,1), (3,1,3,1), (3,1,1,3), (1,3,3,1), (1,3,1,3), (1,1,3,3)\}$, and their weights are $x_1^2x_3^2$ times $1$, $t$, $t^2$, $t^2$, $t^3$ and $t^4$, respectively.
\end{example}

Let $\sigma=(\sigma_1,\ldots,\sigma_m)\in W(m,n)$. We define the transition $\pi_i\,:\,W(m,n) \rightarrow W(m,n)$ by:
\[
\pi_i(\sigma) = \begin{cases} \sigma+ \epsilon_i, & \mbox{if }\sigma_i<n\\
\sigma_{\big(\widehat{i}\big)}\ \mbox{with $\{1\}$ inserted in position } m+1-i, & \mbox{if } \sigma_i=n.
\end{cases}
\]
where $\epsilon_i$ is the vector consisting of all zeros except for a 1 in position $i$, and $\sigma_{\big(\widehat{i}\big)}=\sigma_1,\ldots,\sigma_{i-1},\sigma_{i+1},\ldots,\sigma_n$ is the word $\sigma$ with the element $\sigma_i$ removed. 
The rate of the transition $\pi_i$ is given by 
\[
\rate(\pi_i)=\frac{1}{x_{\sigma_i}}t^{\#\{j<i\,:\,\sigma_j=\sigma_i\}}.
\]
This gives the tableau Markov chain in the special case 
$\lambda=\langle 1^m \rangle$.
As a corollary of \cref{thm:maintheorem}, we obtain:
\begin{cor}
The stationary distribution of the Markov process on the set of words $W(m,n)$ is proportional to 
\[
\wt(\sigma)=x^\sigma t^{\inv(\sigma)}
\]
for each $\sigma\in W(m,n)$. 
\end{cor}

For the one-type case,  the symmetry properties we investigate in this paper hold for zero range processes with a very much wider class of jump rates.  One-type stationary distributions with product form were already obtained in \cite{spitzer-1970}; see also \cite{andjel-1982}, and Chapter 3.4 of \cite{kelly-book} for a broader context involving networks of quasi-reversible queueing servers. Remaining in the totally asymmetric case on the ring, where all jumps from site $i$ go to site $i+1$, cyclically mod $n$, suppose the jump rate at site $i$ when it has $r$ particles is $g_i(r)$. Take any $\alpha>0$ such that for all $i$, $\alpha<\liminf_{r\to\infty} g_i(r)$. Then the distribution given by 
\begin{equation}\label{general-zero range}
\PP(\eta_1, \dots, \eta_n) \;
\propto \; \prod_{i=1}^n \frac{\alpha^{\eta_i}}{\prod_{r=1}^{\eta_i}g_i(r)}
\end{equation}
is invariant for the process. 
To obtain the stationary distribution for a system with a specified number $m$ of particles, one restricts (\ref{general-zero range}) to
the set where $\eta_1+\dots+\eta_n=m$ (in which case the common 
factor $\alpha^m$ can be removed). 

From the product form of the distribution, one immediately observes
symmetry properties stronger than those given 
in Theorem \ref{thm:TAZRP-symmetry}.  For example, if we fix any subset $L$
of the sites, the distribution restricted to 
the sites in $L$ is unchanged under permutation 
of the jump rate functions $(g_i, i\notin L)$.

In our particular case with $g_i(r)=x_i^{-1}(1-t^r)/(1-t)$, 
the stationary distribution can be expressed in the following form.

\begin{prop}
\label{prop:single-species-prob}
The stationary distribution of the mTAZRP on $\operatorname{\TAZRP}(\langle 1^m\rangle,n)$
is given by
\[
\PP(\eta_1,\dots, \eta_n) = \frac{1}{\widetilde{H}_{\langle 1^{m} \rangle}(x_1,\dots,x_n; 1, t)}
\qbinom m{\eta_1,\dots, \eta_n}_t \prod_{i=1}^n x_i^{\eta_i},
\]
when $\eta_1+\dots+\eta_n=m$.
\end{prop}

See \cref{sec:partition} for a discussion of properties of the common denominator $\widetilde{H}_{\lambda}$ in the special case $\lambda=\langle 1^{m} \rangle$.

\subsection{The $t=0$ case}
\label{sec:t0}
At $t=0$, the TAZRP dynamics simplify: at each nonempty site, there is a unique 
possible transition. If there is at least one particle present at $i$, the total jump rate at site $i$ is $x_i^{-1}$. When such a jump occurs, the particle with the
largest label jumps to $i+1$ (cyclically mod $n$).  (If there is more than one equally strong particle at the site,  just one of them jumps.)

On the tableaux side, setting $t=0$ means that the only tableaux
$\sigma\in\Tab(\lambda,n)$ with positive weight
are those with $\quinv(\sigma)=0$. In the tableau Markov chain,
the only transitions from such a $\sigma$ are $R'_u$ for cells
$u$ such that $\down(\sigma,u)=0$. In particular, in the bottom row of $\sigma$, these are the leftmost cells with any given content.

\begin{lemma}
\label{lem:irreducible0}
The tableau Markov chain on $\quinv$-free fillings $\sigma$ with transitions $R'_u$ for each cell $u$ such that $\down(\sigma,u)=0$, is irreducible.
\end{lemma}

\begin{proof}
Let $\lambda$ be a partition and $n$ a positive integer. We show it is possible to get from any $\quinv$-free filling $\sigma\in\Tab(\lambda,n)$ to the filling of $\dg(\lambda)$ whose content is all $n$'s, which we will call $\sigma^{\max}$. 

We show this by constructing a sequence of $\quinv$-free fillings $\sigma=\sigma_0 \prec \sigma_1 \prec \cdots \prec \sigma_m=\sigma^{\max}$, such that $\sigma_k=R'_{u_k}(\sigma_{k-1})$ for some $u_k$ such that $\down(\sigma_{k-1},u_{k})=0$. We define the dominance order $\succ$ on fillings as follows. For two fillings $\tau,\nu\in\Tab(\lambda,n)$, let $u=(r,j)$ be the bottom-most, left-most cell for which $\tau(u)\neq \nu(u)$. we say $\tau\succ \nu$ if $\sum_{j}\tau(r,j) >\sum_j \nu(r,j)$, i.e.~the sum of the entries in row $r$ is greater in $\tau$ than in $\nu$.

Begin by setting $\sigma_0:=\sigma$. Sequentially for $k=1,2,\ldots$, let $u_k=(r,c)$ be the bottom-most, left-most cell such that $\sigma_{k-1}(u_k)<n$. If such a cell does not exist, $\sigma_k=\sigma^{\max}$, and we are done. By construction, every cell in $\arm(u_k)$ has content $n$, so $\down(\sigma_{k-1},u_k)=0$, and so we can set $\sigma_k:=R'_{u_k}(\sigma_{k-1})$. It's a straightforward check that $\quinv(\sigma_k)=0$ and since $\sum_{j}\sigma_k(r,j) =1+\sum_j \sigma_{k-1}(r,j)$, indeed we have $\sigma_k \succ \sigma_{k-1}$. This process necessarily terminates when $\sigma_k=\sigma^{\max}$.
\end{proof}

The weight generating function for this case simplifies nicely. 
Recall that the \emph{complete homogeneous symmetric polynomial} indexed by a partition $\lambda$ is given by $h_{\lambda}(X)=\prod_j h_{\lambda_j}(X)$, where
\[
h_r(X) = \sum_{i_1\leq \cdots \leq i_r} x_{i_1}\cdots x_{i_r}.
\]
The following result is a specialization of \cref{thm:TAZRP-stat-distn} using the identity $\widetilde{H}_\lambda(x_1, \dots, x_n; 1,0)=h_{\lambda'} (x_1,\ldots,x_n)$. The identity follows from the definition of $\widetilde{H}_{\lambda}$, but we give an independent proof based on our tableau formula.

\begin{thm}
Let $\lambda$ be a partition and $n$ a positive integer. Then the weight generating function is
\begin{equation}\label{eq:little h}
\sum_{\substack{\sigma\in \Tab(\lambda,n)\\\quinv(\sigma)=0}} x^{\sigma} = h_{\lambda'} (x_1,\ldots,x_n).
\end{equation}
\end{thm}

\begin{proof}
Let $\lambda$ be a partition with $\lambda'=(\lambda'_1,\ldots,\lambda'_{\ell})$. We claim that for any given set of row contents to fill the shape $\dg(\lambda)$, there is exactly one $\quinv$-free filling in $\Tab(\lambda,n)$. Assuming this, the row contents are independent with respect to each other, hence the generating function over the weights $x^{\sigma}$ of all $\quinv$-free fillings $\sigma$ is equal to the product of the sums over the weights of individual rows:
\[
\sum_{\sigma}x^{\sigma} = \prod_{j=1}^{\ell} \sum_{\sigma_j} x^{\sigma_{(j)}}
\] 
where each $\sigma_{(j)}$ represents the content of row $j$. Now, for a row of length $k$, the sum over the weights of all possible contents is $h_k(x_1,\ldots,x_n)$. Thus, taking the product over contents of rows of lengths $\lambda'_1,\ldots,\lambda'_{\ell}$, we obtain the right hand side of \eqref{eq:little h}. 

We shall now prove the above claim by showing that for any filling $\sigma$ with a given row content, it is always possible to permute the entries within the rows of $\sigma$ to obtain a $\quinv$-free filling with the same row content. Moreover, this filling is unique.

This is easy to see by induction on the number of rows. First, suppose $\dg(\lambda)$ consists of a single row with content $c=\{c_1,\ldots,c_j\}$. Then the unique $\quinv$-free filling of that row has all entries in $c$ arranged in decreasing order from left to right. Next, suppose $\dg(\lambda)$ consists of two rows with contents $c$ and $c'$ respectively, where the entries $c$ in the top row contribute zero to the total $\quinv$. Then there is a unique way to sort the entries $c'$ in the bottom row so that the $\quinv$ of the resulting tableau is 0. One way to do this is to fill the cells in the bottom row from left to right using a greedy algorithm, as follows. If there is no cell above or if the cell above is smaller than or equal to all remaining entries, choose the largest of the remaining entries to fill the cell. Otherwise, fill the cell with the largest of the remaining entries that is smaller than the cell above. With this construction, every triple in the configuration \raisebox{-10pt}{$\qtrip{$a$}{$b$}{$c$}$} within these two rows will either have $a>b\geq c$ (in the first case) or $b\geq c\geq a$ (in the second case), neither of which is a quinv triple.

For example, consider a tableau on two rows with contents $c=\{2,2,3,6\}$ and $c'=\{1,3,4,4,5, \allowbreak 5\}$ in the top and bottom rows, respectively. The unique quinv-free filling of the top row is to have all entries in decreasing order. For the bottom row, our algorithm  will produce the following $\quinv$-free filling:
\[
\tableau{6&3&2&2\\5&\ &\ &\ &\ &\ } \rightarrow \tableau{6&3&2&2\\5&1&\ &\ &\ &\ }\rightarrow \tableau{6&3&2&2\\5&1&5&\ &\ &\ }\rightarrow
\]
\[ \tableau{6&3&2&2\\5&1&5&4&\ &\ }\rightarrow \tableau{6&3&2&2\\5&1&5&4&4&\ }\rightarrow \tableau{6&3&2&2\\5&1&5&4&4&3} \]

Now, suppose $\dg(\lambda)$ has $m$ rows, and let $1< k\leq m$. For a filling $\sigma$ and some subset of rows $S\subseteq[m]$, denote the filling of the rows $S$ by $\sigma_S$. Then $\sigma_{[k-1]}$ and $\sigma_{[k,m]}$ are the fillings of the rows $1, \dots, k-1$ and $k,\ldots,m$, respectively. Suppose that the entries of $\sigma_{[k,m]}$ are already sorted so that $\quinv(\sigma_{[k,m]})=0$. Define $\quinv'(T)$ to be the contribution to $\quinv$ of all triples except for the degenerate triples in the topmost row. Then we may write 
\[\quinv(\sigma)=\quinv'(\sigma_{[k-1]})+\quinv'(\sigma_{\{k-1,k\}}).\]
We can now apply the argument in the paragraph above for two-row fillings to sort row $k-1$ so as to arrive at the unique filling $\sigma_{\{k-1,k\}}'$ such that $\quinv'(\sigma_{\{k-1,k\}}')=0$. Letting $\sigma'$ be the filling $\sigma$ with row $k-1$ sorted accordingly, we can then write 
\[\quinv(\sigma')=\quinv'(\sigma'_{[k-2]})+\quinv'(\sigma'_{\{k-2,k-1\}}),\]
since $\quinv(\sigma'_{[k-1,m]})=\quinv'(\sigma_{[k,m]})+\quinv'(\sigma_{\{k-1,k\}})=0$. Continuing this process of sorting the rows from top to bottom until we reach row 1, we end up with a filling with the same row content of $\sigma$, and $\quinv$ equal to 0.
\end{proof}

\section{Partition function}\label{sec:partition}
Both the tableau Markov chain and the mTAZRP itself are ergodic finite-state-space Markov chains with parameters $t$ and $x_1, \dots, x_n$. 

For such a chain, suppose that we can write the stationary probability of each state $s$ as $\pi(s)=w(s)/Z$, 
where $w(s)$ is a polynomial in $t, x_1, \dots, x_n$ with integer coefficients. It follows that $Z=\sum_s w(s)$. Then we say that $Z$ is a \emph{partition function} for the chain. 

We say that the partition function $Z$ is \emph{reduced} if $\gcd_s w(s)=1$ (i.e.\ if the stationary weights $w$ are in their lowest
terms).

From \cref{thm:maintheorem},
we know that the tableau Markov chain 
on $\Tab(\lambda, n)$ 
has stationary probabilities proportional to the weights 
$\wt(\sigma)=t^{\quinv(\sigma)}x^{\sigma}$,
and hence has partition function 
$\sum_{\sigma\in\Tab(\lambda, n)} \wt(\sigma)=
\widetilde{H}_\lambda(X;1,t)$.
In fact, each of these stationary weights 
$\wt(\sigma)$ is a monomial, and their greatest common
denominator is $1$
(e.g.\ for $n\geq 2$,
the monomials $x_1^{|\lambda|}$ and $x_2^{|\lambda|}$ both occur as stationary weights).
Thus we have the following:
\begin{thm}
\label{thm:pf}
The partition function
\[
Z_{\Tab}(X;t) := \widetilde{H}_\lambda(X;1,t)
\]
of the tableau Markov chain on $\Tab(\lambda,n)$ 
is reduced. 
\end{thm}

At $q=1$, the modified Macdonald polynomials factorize and lose much of their complexity. Here is a summary of some of the properties at $q=1$:
\begin{enumerate}

\item $\widetilde{H}_{\lambda}(X;1,t)=\widetilde{H}_{\lambda'}(X;t,1)$

\item 
\label{it:Ht-factor}
If $\lambda=(\lambda_1,\ldots,\lambda_k)$, then
\begin{equation}
\label{eq:Ht-factor}
\widetilde{H}_{\lambda}(X;1,t) = \prod_{j=1}^{\ell(\lambda')} \widetilde{H}_{\big\langle 1^{\lambda'_j} \big\rangle}(X;1,t).
\end{equation}

\item 
\label{it:Ht-monomialsum}
Recall that $m_{\mu}(X)$ is the monomial symmetric function indexed by $\mu$ given by
\[
m_\mu(X) = \sum_\alpha x_1^{\alpha_1}x_2^{\alpha_2}\cdots,
\]
where the sum runs over weak compositions $\alpha = (\alpha_1, \alpha_2, \dots)$ that sort to $\mu$.
If $r$ is a positive integer, $\widetilde{H}_{\langle 1^r \rangle}$ expands in terms of monomial symmetric functions as
\begin{equation}
\label{eq:Ht-monomialsum}
\widetilde{H}_{\langle 1^r \rangle}(X;1,t) = \sum_{\mu \vdash r} \qbinom {r}{\mu}_t m_{\mu}(X),
\end{equation}
where $\mu\vdash r$ represents the set of partitions $\mu$ whose parts sum to $r$ and the notation $\qbinom {r}{\mu}_t$ stands for the $t$-multinomial coefficient $\qbinom {r}{\mu_1,\ldots,\mu_{\ell} }_t$ for $\mu=(\mu_1,\ldots,\mu_{\ell})$. This follows because for fillings $\sigma$ with content $\mu$ of a single row, $\maj(\sigma)=0$, and so the weight generating function is simply the inversion generating function over words with content $\mu$.

Thus, we have that 
\[
Z_{\Tab} = \prod_{j=1}^k  \sum_{\mu \vdash \lambda_j} \qbinom {\lambda_j}{\mu}_t m_{\mu}. 
\]

\item 
\label{it:Ht-recursion}
As a consequence of the monomial expansion in \cref{it:Ht-monomialsum}, we can expand the modified Macdonald polynomial in powers of $x_1$ as
\begin{align}
\widetilde{H}_{\langle 1^r \rangle}(x_1,\ldots,x_n;1,t) &= \sum_{i=1}^r \qbinom{r}{i}x_1^i \sum_{\mu \vdash r-i} \qbinom{r-i}{\mu}_t m_{\mu}(x_2,\ldots,x_n) \nonumber\\
&= \sum_{i=0}^r\qbinom {r}{i}_t x_1^i \widetilde{H}_{\langle 1^{r-i} \rangle}(x_2,\ldots,x_n;1,t).\label{eq:expand}
\end{align}

\end{enumerate}

Suppose $\lambda = \langle i_1^{m_1}, i_2^{m_2}, \dots, i_j^{m_j} \rangle$, 
where $i_1 < i_2 < \cdots < i_j$ and $m_j > 0$ for all $j$. Then, we define the 
\emph{compressed partition} 
$\lambdac$ to be the partition $\lambdac = \langle 1^{m_1}, 2^{m_2}, \dots, j^{m_j} \rangle$. 
If $\lambda$ has contiguous parts starting from $1$, 
then clearly $\lambdac = \lambda$.
Also, notice that the conjugate of $\lambdac$ is a strict partition. Analogously, we also define $w^c\in \TAZRP(\lambda^c,n)$ to be the compressed composition naturally corresponding to $w\in \TAZRP(\lambda,n)$. 
The following result easily follows from the fact that  
every row-length of $\lambda^c$ is also a row-length 
of $\lambda$:

\begin{prop}
For any partition $\lambda$, $\widetilde{H}_{\lambda}(X;1,t)$ is divisible by $\widetilde{H}_{\lambdac}(X;1,t)$.
\end{prop}

For example, let $\lambda = (7,4,4,2,2,2,1)$ so that
\[
\widetilde{H}_{\lambda}(X;1,t) = \widetilde{H}_{(1)}(X;1,t)^3
\widetilde{H}_{(1^3)}(X;1,t)^2 \widetilde{H}_{(1^6)}(X;1,t)^1 \widetilde{H}_{(1^7)}(X;1,t)^1.
\]
Then $\lambdac = (4,3,3,2,2,2,1)$ and 
\[
\widetilde{H}_{\lambdac}(X;1,t) = 
\widetilde{H}_{(1)}(X;1,t)
\widetilde{H}_{(1^3)}(X;1,t) \widetilde{H}_{(1^6)}(X;1,t) \widetilde{H}_{(1^7)}(X;1,t).
\]

For the mTAZRP chain on a compressed partition, we have verified the following analog of \cref{thm:pf} for many partitions and system sizes, and suspect it is true in general.
Note that such a result does not hold for the multispecies ASEP~\cite{martin-2020}.

\begin{conj}
\label{conj:compressed}
Suppose $\lambda = \lambdac$. Then the partition function of the mTAZRP chain on $\TAZRP(\lambda, n)$ given by $\widetilde{H}_{\lambda}(X;1,t)$ is reduced.
\end{conj}

Now, we consider the mTAZRP chain on $\TAZRP(\lambda, n)$ for a generic partition $\lambda$. In this case, we see that the particle types of $\lambda$ might as well be those of $\lambdac$ since only the relative orderings of the particle types matter. Thus the stationary weights on $\TAZRP(\lambda,n)$ are proportional to the corresponding stationary weights on $\TAZRP(\lambdac,n)$. We therefore have the following result:
\begin{prop}
\label{prop:reduced pf}
Fix $\lambda$ and $n$. For each state $w\in\TAZRP(\lambda,n)$, the sum over stationary weights
\[
\wt(w)=\sum_{\substack{\sigma\in\Tab(\lambda,n)\\\proj(\sigma)=w}}  \wt(\sigma)
\]
is divisible by the constant $\widetilde{H}_{\lambda}(x_1,\ldots,x_n;1,t)/ \widetilde{H}_{\lambdac}(x_1,\dots,x_n;1,t)$,
and the mTAZRP on $\TAZRP(\lambda, n)$ has a partition function 
given by
\begin{equation}\label{eq:ZT}
Z_{\TAZRP(\lambda,n)} = \widetilde{H}_{\lambdac}(x_1,\dots,x_n;1,t).
\end{equation}
\end{prop} 

Notice that if \cref{conj:compressed} holds, then the partition 
function (\ref{eq:ZT}) for the mTAZRP on $\TAZRP(\lambda, n)$ is
in fact reduced.

While the result of Proposition \ref{prop:reduced pf} is trivial to see probabilistically, we do not have a combinatorial proof. It would be very interesting to prove this combinatorially.

We now formulate a refined conjecture which would lead to a combinatorial proof of \cref{prop:reduced pf} in the special case when $\lambda$ is a strict partition.
Let $\lambda$ be a strict partition and let $\sigma\in\Tab(\lambdac,n)$ be a filling of 
$\dg(\lambdac)$. 
Notice that $\sigma$ can be embedded as a subfilling of $\dg{\lambda}$ in a natural way. 
For example, if $\lambda = (5,3,2)$, then $\lambdac = (3,2,1)$. Let $\sigma\in\Tab(\lambdac,3)$ be
\[
\sigma = \raisebox{8pt}{\tableau{
3 \\
3 & 2 \\
1 & 1 & 2 }}
\]
Then
$\sigma$ sits in $\dg((5,3,2))$ as a subfilling as
\[
\tableau{
\cdot \\
\cdot \\
3 & \cdot \\
3 & 2 & \cdot \\
1 & 1 & 2 }
\]
We denote the set of fillings of $\dg(\lambda)$ obtained by extending $\sigma$ by $\Ext_{\lambda}(\sigma)$ (every such filling corresponds to a filling of the skew shape $\dg(\lambda\backslash\lambdac)$). 
We can now state our conjecture refining \cref{prop:reduced pf}.

\begin{conj}
\label{conj:refined pf}
Let $\lambda$ be a strict partition, $n \in \mathbb{N}$ and let $\sigma\in\Tab(\lambdac,n)$. Then 
\[
\frac{1}{\wt(\sigma)}\sum_{T \in \Ext_{\lambda}(\sigma)} \wt(T) = \frac{\widetilde{H}_{\lambda}(x_1,\dots,x_n;1,t)}{\widetilde{H}_{\lambdac}(x_1,\dots,x_n;1,t)}.
\]
\end{conj}

We now show that a proof of \cref{conj:compressed,conj:refined pf} will give a combinatorial proof of \cref{prop:reduced pf}. First, note that it suffices to show that
\[
\gcd_{w \in \TAZRP(\lambda,n)} \wt(w) = \frac{\widetilde{H}_{\lambda}(x_1,\dots,x_n;1,t)}{\widetilde{H}_{\lambdac}(x_1,\dots,x_n;1,t)}.
\]
The stationary probability of a configuration $w \in \TAZRP(\lambda,n)$ is given according to \cref{lem:lumping} by summing over the weights of fillings of $\sigma\in\Tab(\lambda,n)$, where $\sigma$ ranges over all fillings with a fixed bottom row $b$ such that $\proj(b)=w$. The same is true for $w^{\text{c}} \in \TAZRP(\lambdac,n)$, with the same $b$ because $\proj(T)=\proj(\sigma)$ for any $T\in\Ext_{\lambda}(\sigma)$, as the map $f$ depends only on the bottom row of a filling, and the bottom row of $T$ is equal to that of $\sigma$ by definition. 

We then have
\[
\wt(w) = \sum_{\substack{\sigma \in \Tab(\lambda,n) \\ \proj(\sigma)=w}} \wt(\sigma).
\]
We now refine the above by summing over fillings $\dg(\lambdac)$ with the same bottom row:
\begin{align*}
\wt(w) =& \sum_{\substack{\sigma \in \Tab(\lambdac,n) \\ \proj(\sigma)=w}} \sum_{T \in \Ext_{\lambda}(\sigma)} \wt(T), \\
=& \sum_{\substack{\sigma \in \Tab(\lambdac,n) \\ \proj(\sigma)=w}} \wt(\sigma) \frac{\widetilde{H}_{\lambda}(x_1,\dots,x_n;1,t)}{\widetilde{H}_{\lambdac}(x_1,\dots,x_n;1,t)}, \\
=& \wt(w^{\text{c}}) \frac{\widetilde{H}_{\lambda}(x_1,\dots,x_n;1,t)}{\widetilde{H}_{\lambdac}(x_1,\dots,x_n;1,t)},
\end{align*}
where we have used \cref{conj:refined pf} in the second line.
Now, \cref{conj:compressed} completes the argument.

\begin{example}
Let $\lambda = (3,2)$ so that $\lambdac = (2,1)$ with $n = 2$. Consider the following tableau in $\sigma\in\Tab(\lambdac,2)$, with $\proj(\sigma)=(1|2) \in\TAZRP(\lambdac,2)$:
\[
\sigma = \tableau{2\\2&1}
\]
Using \cref{thm:maintheorem0}, we have $\wt(\sigma)=t x_1 x_2^2$.
The tableaux and weights in $\Ext(\sigma)$ are shown in the following table:
\[
\begin{array}{c|c||c|c||c|c||c|c}
\text{Filling} & \text{Weight} & \text{Filling} & \text{Weight} & \text{Filling} & \text{Weight} & \text{Filling} & \text{Weight} \\
\hline&&&&&& \\[-1.5ex]
\tableau{1\\2&1\\2&1} & t x_1^3 x_2^2 & 
\tableau{1\\2&2\\2&1} & t x_1^2 x_2^3 & 
\tableau{2\\2&1\\2&1} & t^2 x_1^2 x_2^3 & 
\tableau{2\\2&2\\2&1} & t x_1 x_2^4 
\end{array}
\]
The total contribution of these fillings to the weight of $(2|3)\in\TAZRP(\lambda,2)$ is
\[
\wt((2|3))=t x_1 x_2^2 \left(x_1^2 + (1+t) x_1 x_2 + x_2^2\right),
\]
where the prefactor is equal to $\wt(\sigma)$ and the factor in parenthesis is precisely $\frac{\widetilde{H}_{(3,2)}}{\widetilde{H}_{(2,1)}} = 
\widetilde{H}_{(1,1)}(x_1,x_2;1,t)$, in accordance with \cref{conj:refined pf}.
\end{example}

\section{Observables}
\label{sec:obs}

\subsection{Symmetries in probabilities of configurations in the first $\ell$ sites}
\label{sec:symmetry}

Fix a partition $\lambda$ and a positive integer $n$, and consider the TAZRP of type $\lambda$ on $n$ sites. Fix $\ell$ with 
$1\leq \ell\leq n-1$. We will consider the probability 
of observing some specific configuration $w$ on the sites $1,\dots, \ell$. 
Formally, $w$ can be seen as a configuration of TAZRP$(\mu, \ell)$ 
for some partition $\mu$ consisting of a subset of the parts of 
$\lambda$. 

Now we write $\overline{w}$ for the set of configurations
of TAZRP$(\lambda,n)$ whose restriction to $1,\dots,\ell$ is given
by $w$.

\begin{example}
\label{ex:w} 
Let $\lambda=(2,2,1,1)$, $n=4$, $\ell=2$, and $w=(1|2)$. (Then  $\mu=(2,1)$). The configurations of $\TAZRP(\lambda,n)$ contributing to $\mathbb{P}_{\lambda,n}(\overline{w})$ are
\[
\{(1|2|21|\cdot), (1|2|2|1), (1|2|1|2), (1|2|\cdot|21)\}.
\]
In this case, the coefficient of $t^3$ in $P_{\lambda,n}(\overline{w})$ is $x_1x_2(x_2^2+x_1x_2)m_2 + (x_1+2x_2)m_3 + (4x_1x_2+3x_2^2)m_{11} + (3x_1+4x_2)m_{21} +  m_{31} + 2m_{22})/\widetilde{H}_{\lambda}(x_1,x_2,x_3,x_4;1,t)$, where $m_{\mu}=m_{\mu}(x_3,x_4)$ is the monomial basis evaluated at the variables $x_3,x_4$.
\end{example}

The main result we prove in this section is the following.

\begin{prop}
\label{thm:symm probs}
Fix a partition $\lambda$ and positive integer $n$.
For any $0 \leq \ell\leq n$ and configuration $w$ on the first $\ell$ sites of the TAZRP of type $\lambda$ on $n$ sites, 
the probability 
$\mathbb{P}_{\lambda,n}(\overline{w})$
is symmetric in the variables $\{x_{\ell+1},\ldots,x_n\}$.
\end{prop}
Note that when $\ell=0$, $w$ is the empty configuration and this is just the statement that the partition function is symmetric in the variables $\{x_1,\ldots,x_n\}$.

For instance, in \cref{ex:w} the coefficient $[t^3]P_{\lambda,n}(\overline{w})$ where $\lambda=(2,2,1,1)$, $n=4$, and $w=(1|2)$ is symmetric in the variables $x_3,x_4$. 

\cref{thm:symm probs} proves \cref{thm:TAZRP-symmetry} and will follow as an immediate corollary of \cref{lem:C}.

We write $\wt(\overline{w};\lambda,n) = \widetilde{H}_\lambda(x_1,\dots,x_n;1,t)  \mathbb{P}_{\lambda,n}(\overline{w})$ for the corresponding stationary weight. 
Let $\sigma\vert_{\ell}$ denote the restriction of $\sigma$ to the cells with content $\leq \ell$. Then we have
\begin{equation}\label{eq:Pr}
\wt(\overline{w};\lambda,n) = \sum_{\substack{\sigma\in\Tab(\lambda,n)\\\proj(\sigma\vert_{\ell}) = w}} t^{\quinv(\sigma)}x^{\sigma}.
\end{equation}

\begin{example}
Continuing our running example from above, $\wt(\overline{w};\lambda,n)$ is given by the sum over the following fillings.
\begin{center}
\begin{tikzpicture}[scale=0.5]
\node at (-4,0) {$(1|2|2|1):$};
\node at (0,0) {\tableau{\ast&\ast\\ 2&3&1&4}};
\node at (4,0) {\tableau{\ast&\ast\\ 2&3&4&1}};
\node at (8,0) {\tableau{\ast&\ast\\ 3&2&1&4}};
\node at (12,0) {\tableau{\ast&\ast\\ 3&2&4&1}};

\node at (0,-3) {\tableau{\ast&\ast\\ 2&4&1&3}};
\node at (4,-3) {\tableau{\ast&\ast\\ 2&4&3&1}};
\node at (8,-3) {\tableau{\ast&\ast\\ 4&2&1&3}};
\node at (12,-3) {\tableau{\ast&\ast\\ 4&2&3&1}};

\node at (0,-6) {\tableau{\ast&\ast\\ 2&3&1&3}};
\node at (4,-6) {\tableau{\ast&\ast\\ 2&3&3&1}};
\node at (8,-6) {\tableau{\ast&\ast\\ 3&2&1&3}};
\node at (12,-6) {\tableau{\ast&\ast\\ 3&2&3&1}};

\node at (0,-9) {\tableau{\ast&\ast\\ 2&4&1&4}};
\node at (4,-9) {\tableau{\ast&\ast\\ 2&4&4&1}};
\node at (8,-9) {\tableau{\ast&\ast\\ 4&2&1&4}};
\node at (12,-9) {\tableau{\ast&\ast\\ 4&2&4&1}};

\node at (-4,-3) {$(1|2|1|2):$};
\node at (-4,-6) {$(1|2|21|\cdot):$};
\node at (-4,-9) {$(1|2|\cdot|21):$};

\end{tikzpicture}
\end{center}

\end{example}

We refine \eqref{eq:Pr} further. We first introduce the following notation. 

\begin{defn}
Let $S\subseteq \dg(\lambda)$ be a subset of cells, and let $0\leq\ell\leq n$ be fixed. Given two functions 
\[
\sle: S \rightarrow [\ell] \quad \text{and} \quad \sg:\dg(\lambda)\backslash S \rightarrow \{\ell+1,\ldots,n\},
\]
we define the filling 
$\sigma=(\sle,\sg)\in\Tab(\lambda,n)$
as the function which agrees with $\sle$ on $S$ and which
agrees with $\sg$ on $\dg(\lambda)\setminus S$. 
Then $\sigma$ has
all cells in $S$ filled with entries $\leq \ell$ and all cells outside of $S$ filled with entries $>\ell$.
\end{defn}

\begin{example}
For $\lambda=(3,2,2)$,  $S = \{(1,1), (1,3), (2,1), (2,2)\}$, $\ell=2$, the tableau $\sigma=(\sle,\sg)$ decomposes as below.
\begin{center}
\begin{tikzpicture}[scale=0.5]
\node at (-2,0) {$\sigma=$};
\node at (0,0) {\tableau{4\\1&2&3\\2&4&2}};
\node at (4.5,0) {$\sle=$};
\node at (7,0) {\tableau{\greysq\\&&\greysq\\&\greysq&}};
\node at (7,0) {\tableau{\\1&2&\\2&&2}};
\node at (12,0) {$\sg=$};
\node at (14.5,0) {\tableau{&\\\greysq&\greysq&\\\greysq&&\greysq}};
\node at (14.5,0) {\tableau{4\\&&3\\&4&}};
\end{tikzpicture}
\end{center}
\end{example}

For any fixed $\ell$, one can write the set of all fillings $\Tab(\lambda,n)$ as the following disjoint union:
\[
\Tab(\lambda,n) = \bigcup_{S\subseteq \dg(\lambda)}\ \bigcup_{\sle:S\rightarrow [\ell]} \Big\{ \sigma=(\sle,\sg): \sg:\dg(\lambda)\backslash S \rightarrow \{\ell+1,\ldots,n\}\Big\}.
\]

We can similarly refine \eqref{eq:Pr}. Define $\dg(\lambda;w)\subseteq\dg(\lambda)$ to be a set of cells \emph{compatible} with $w$, meaning that $\dg(\lambda;w)$ restricted to the bottom row must have the same type as $w$.
\begin{equation}\label{eq:refined}
\wt(\overline{w};\lambda,n) = \sum_{S\subseteq \dg(\lambda;w)}\ \sum_{\substack{\sle :S \rightarrow [\ell]\\\proj(\sle)=w}}\ \sum_{\sg:\dg(\lambda)\backslash S\rightarrow \{\ell+1,\ldots,n\}} t^{\quinv(\sle,\sg)}x^{(\sle,\sg)}.
\end{equation}

To simplify notation, for any $S\subseteq \dg(\lambda)$ and a filling $\nu: S \rightarrow [\ell]$, define
\begin{equation}\label{eq:C}
C_{\lambda,\nu}=C_{\lambda,\nu}^{(\ell)} (X;t) = \sum_{\sg:\dg(\lambda)\backslash S\rightarrow \{\ell+1,\ldots,n\}} t^{\quinv(\nu,\sg)}x^{(\nu,\sg)}
\end{equation}

We will prove the following.

\begin{lemma}\label{lem:C}
For any $S\subseteq \dg(\lambda)$ and a fixed filling $\nu:S\rightarrow [\ell]$, $C_{\lambda,\nu}(x_1,\ldots,x_n;t)$ is symmetric in the variables $\{x_{\ell+1},\ldots, x_n\}$.
\end{lemma}

We prove this by writing $C_{\lambda,\nu}(x_1,\ldots,x_n;t)$ as a sum over Lascoux--Leclerc--Thibon (LLT) polynomials in the variables $\{x_{\ell+1},\ldots,\allowbreak x_n\}$, which are known to be symmetric \cite{LLT}. Recall that the \emph{LLT polynomials}, denoted $G_{\boldsymbol{\gamma}}(x_1,\dots,x_n; t)$ are a family of symmetric polynomials indexed by a tuple of skew shapes $\boldsymbol{\gamma}=(\gamma_1,\ldots,\gamma_k)$ (in our case, the skew shapes will be ribbons). Let $\LLT(\boldsymbol{\gamma})$ be the set of semistandard fillings of the skew shapes. Then $G_{\boldsymbol{\gamma}}(X;t)$ can be written as a weighted sum over $\LLT(\boldsymbol{\gamma})$ with the statistic $t$ counting certain inversions in the fillings, which we shall refer to as \emph{LLT inversions}~\cite{LLT}.
  In fact, we will use essentially the same proof strategy that was used in the prequel to this article to show our tableaux formula for $\widetilde{H}_{\lambda}(x_1,\ldots,x_n;t)$ is symmetric in the variables $\{x_1,\ldots,x_n\}$ \cite[\S 4]{AMM20}, which was in turn adapted from the proof of symmetry of the HHL formula for $\widetilde{H}_{\lambda}$ in \cite[\S 5]{HHL05}. See \cite[\S 4]{AMM20} for the relevant definitions.

\begin{defn}
Let $\lambda$ be a partition and $S\subseteq \dg(\lambda)$ be any configuration of cells. The \emph{descent set} $D_S(\sigma)$ of a filling $\sigma$ of $S$ is
the restriction of the descent set $D(\sigma)$ to the cells $S$:
\[ 
D_S(\sigma) = \{x\in S\ :\ \South(x)\in S,\ \sigma(x)>\sigma(\South(x))\}.
\]
When $S=\dg(\lambda)$, we will simply write $D(\sigma)$.
\end{defn}

\begin{defn}\label{def:attacking}
Let $\lambda$ be a partition and $S\subseteq \dg(\lambda)$ be any subset of cells. An \emph{attacking inversion} of a filling $\sigma$ of $S$ is a pair of cells $u,v\in S$ 
such that $\sigma(u)>\sigma(v)$, and $u,v$ are in one of two configurations:
\begin{itemize}
\item (Type A) $u=(r,i)$ and $v=(r-1,j)$:\ \raisebox{-10pt}{\begin{tikzpicture}[scale=0.5]\cellL00{}{$u$}\node at (1,-.5) {$\ldots$}; \cellL13{}{$v$}\end{tikzpicture}}\ , or 
\item (Type B) $u=(r,j)$ and $v=(r,i)$:\ \raisebox{-5pt}{\begin{tikzpicture}[scale=0.5]\cellL00{}{$v$}\node at (1,.5) {$\ldots$}; \cellL03{}{$u$}\end{tikzpicture}}\ ,
\end{itemize}
for some row $r$.
Borrowing notation from \cite{AMM20}, we let $\widehat{\inv}(\sigma)$ denote the number of attacking inversions in $\sigma$.
\end{defn}

\begin{defn}
Let $\lambda$ be a partition. The \emph{arm length} of a cell $u=(r,j)\in \dg(\lambda)$, denoted by $\widehat{\A}(u)$, is 
\[
\widehat{\A}(u) =\Big| \{(r-1,i)\in \dg(\lambda)\ |\ i>j \}\Big|.
\] 
One way to think of $\widehat{\A}(u)$ is as the set of all cells $v\in \dg(\lambda)$ that could potentially form attacking inversions of type B with $u$.  The arm length of a filling $\sigma$ of $\dg(\lambda)$ is the sum of the arm lengths of all the cells in $D(\sigma)$:
\[
\widehat{\A}(\sigma) = \sum_{u\in D(\sigma)} \widehat{\A}(u).
\]
\end{defn}

Although immediate from the definition, the following lemma is helpful.

\begin{lemma}\label{lem:arm}
Let $\lambda$ be a partition. If $\sigma,\sigma'\in\Tab(\lambda)$ have $D(\sigma)=D(\sigma')\subseteq \dg(\lambda)$, then also $\widehat{\A}(\sigma)=\widehat{\A}(\sigma')$.
\end{lemma}

Thus from now on we may write $\widehat{\A}(\lambda,D):=\widehat{\A}(\sigma)$ where $\sigma$ is a filling of $\dg(\lambda)$ and $D=D(\sigma)$.

The way we relate these new statistics to our tableaux formula is via the following lemma that was proved in the prequel to this article.

\begin{lemma}[{\cite[Lemma 4.6]{AMM20}}]\label{lem:quinv}
Let $\lambda$ be a partition and let $\sigma\in\Tab(\lambda)$. Then 
\[
\quinv(\sigma) = \widehat{\inv}(\sigma) - \widehat{\A}(\sigma) = \widehat{\inv}(\sigma)-\widehat{\A}(\lambda,D(\sigma)).
\]
\end{lemma}

\begin{lemma}
\label{lem:descent}
Fix a partition $\lambda$, $\ell\leq n$, and a sub-tableau $\nu:S\rightarrow [\ell]$ for some subset of cells $S\subseteq \dg(\lambda)$. Let $\sg:\dg(\lambda)\backslash S \rightarrow \{\ell+1,\ldots,n\}$. 
Then the number of attacking inversions of a filling $\sigma=(\nu,\sg)$ decomposes as
\begin{equation}\label{eq:inv}
\widehat{\inv}(\sigma) = \widehat{\inv}(\sg)+
I(\nu,\lambda),
\end{equation}
where $I(\nu,\lambda)$ is independent of $\sigma^>$. 
\end{lemma}

\begin{proof}

Each pair $u,v$ that forms an attacking inversion (where $\sigma(u)>\sigma(v)$ and $u,v$ are in one of the two configurations from \cref{def:attacking}) belongs in one of three categories, splitting $\widehat{\inv}(\sigma)$ into three components.
If $u,v \not\in S$, the pair contributes to $\widehat{\inv}(\sg)$. If $u,v\in S$, the pair contributes to $\widehat{\inv}(\nu)$. Otherwise, since we assume $\sigma(u)>\sigma(v)$, we must have $u\not\in S$ and $v\in S$, so the pair contributes to the third summand. The second summand is a function of $\nu$, and the third summand is a function of the diagram $S$ with relation to $\dg(\lambda)$, and so we write both combined as $I(\nu,\lambda)$.
\end{proof}

We now construct a certain weight-preserving bijection $\sh$ between tableaux with a fixed descent set and LLT fillings of ribbon shapes. Fix a partition $\lambda$, a set of cells $S\subseteq \dg(\lambda)$ and a descent set $D\subseteq S$. 
Let $\Tab(S,D)$ denote fillings of $S$ with descent set $D$.  Set $k = |\ell(\lambda)|$ and let $\boldsymbol{\gamma}=(\gamma_1,\ldots,\gamma_{k})$ be a tuple of ribbons.

We construct a map $\sh: \Tab(S,D) \to \LLT(\boldsymbol{\gamma})$ as follows. 
For $1 \leq j \leq k$, construct the ribbon $\gamma_{k-j+1}$ from the cells in column $j$ of $S$. Each cell $x=(r,j)\in S$ is mapped to the cell $y(x)$ in $\gamma_{k-j+1}$ such that:
\begin{itemize}
\item $y(x)$ is on diagonal $r-1$,
\item if $\South(x)\not\in S$, then $y(x)$ has no adjacent cells below or to its right,
\item if $\South(x)\in S$ and $x\in D$, then $y(x)$ is directly above $y(\South(x))$ (i.e.~$\South(y(x))=y(\South(x))$), and
\item if $\South(x)\in S$ and $x\not\in D$, then $y(x)$ is directly to the left of $y(\South(x))$.
\end{itemize} 
Now, a filling $\sigma$ of $S$ is mapped to a SSYT of shape of $\boldsymbol{\gamma}$ by sending the content $\sigma(x)$ to the cell $y(x)$. With slight abuse of notation, we write $\boldsymbol{\gamma} = \sh(S,D)$.

By the construction above, the filling $\sh(\sigma)$ of $\boldsymbol{\gamma}$ will be weakly increasing along rows from left to right and strictly increasing along columns from bottom to top. Moreover, it is quite straightforward to check that for a filling $\sigma$ of $S$, two cells $u,v\in S$ contribute to $\widehat{\inv}(\sigma)$ if and only if they form an LLT inversion in $\sh(\sigma)$.

\begin{lemma}
\label{lem:lltA}
Let $\sigma$ be a filling of $S\subseteq \dg(\lambda)$ with the integers in $\{\ell+1,\ldots,n\}$, $D=D_S(\sigma)$ be its descent set, and $\boldsymbol{\gamma}=\sh(S,D)$. Then
\[
\sum_{\substack{\sigma:S \rightarrow   \{\ell+1,\ldots,n\}\\ D(\sigma)=D}} x^{\sigma}t^{\widehat{\inv}(\sigma)} = \LLT_{\boldsymbol{\gamma}}(x_{\ell+1},\ldots,x_n).
\] 
\end{lemma}

We will not provide a proof, but will instead direct the reader to \cite[Section 4]{AMM20} where this is proved for the case $S=\dg(\lambda)$, and naturally generalizes to this setting as well. We observe that because the filling of $\nu$ contains only integers $\leq \ell$ and the filling of $\sigma$ contains only integers $\geq \ell+1$, the diagram $\sh(S,D)$ that we have described is equivalent to the diagram $\sh(\dg(\lambda),D(\nu,\sigma))$ when restricted to the cells corresponding to $S$, up to relative position.

We shall now use the above to write $C_{\lambda,\nu}$ as a sum over LLT polynomials, from which their symmetry follows.

\begin{proof}[Proof of \cref{lem:C}]
By \cref{lem:arm,lem:quinv}, we can write

\begin{align*}
C_{\lambda,\nu}(x_1,\ldots,x_n;t)& = \sum_{D\subseteq (\dg(\lambda)\backslash S)^+}\quad \sum_{\substack{\sg:\dg(\lambda)\backslash S \rightarrow \{\ell+1,\ldots,n\}\\ D(\sg)=D}} x^{(\nu,\sg)}t^{\widehat{\inv}(\nu,\sg)-\widehat{\A}(\nu,\sg)}\\
&=x^{\nu}t^{\widehat{\inv}(\nu,\lambda)} \sum_{D\subseteq (\dg(\lambda)\backslash S)^+} t^{-\widehat{\A}(\lambda,D)}
\sum_{\substack{\sg:\dg(\lambda)\backslash S \rightarrow   \{\ell+1,\ldots,n\}\\ D(\sg)=D}} x^{\sg}t^{\widehat{\inv}(\sg)}
\end{align*}
where for a subset of cells $A$, we use the notation $A^+:=\{x\in A\ :\ \South(x)\in A\}\subseteq A$ (this is to be consistent with the fact that any element of a descent set must have an existing cell south of it). 

By \cref{lem:lltA}, we get the expansion
\[
C_{\lambda,\nu}(x_1,\ldots,x_n) = x^{\nu}t^{\widehat{\inv}(\nu,\lambda)} \sum_{D\subseteq (\dg(\lambda)\backslash S)^+} t^{-\widehat{\A}(\lambda,D)} \LLT_{\sh(S,D)} (x_{\ell+1},\ldots,x_n),
\]
which, given that its first component is in the variables $x_1,\ldots,x_\ell$, implies that $C_{\lambda,\nu}$ is indeed symmetric in the variables $\{x_{\ell+1},\ldots,x_n\}$.
\end{proof}

\subsection{A stronger symmetry property in the case $t=0$}
\label{sec:stronger-symmetry}

In the case $t=0$ we have the following stronger 
symmetry property, which applies to the paths
of the process and not just its stationary distribution.

\begin{thm}
\label{thm:stronger-symmetry}
Fix some $\ell$ with $1<\ell<n$. 
Consider the mTAZRP process, started at time $s=0$ from any initial condition in which sites $\ell+1, \dots, n$ are empty. 

The distribution of the path of the process on time interval 
$(0,\infty)$, restricted to the sites $1,\dots, \ell$, 
is unchanged under permutation of the
parameters $x_{\ell+1}, \dots, x_n$. 
\end{thm}

Observe that the conclusion of this result implies that of Theorem \ref{thm:TAZRP-symmetry} (for the case $t=0$).
From the result above, we have in particular that 
the distribution at any fixed time $s$ of the configuration
on sites $1,\dots,\ell$ is invariant under permutation of the 
parameters $x_{\ell+1}, \dots, x_n$. However, that time-$s$ 
distribution converges as $s\to\infty$ to the stationary distribution. 
Hence the stationary distribution itself must also be 
invariant in the same way, which is the content of 
Theorem \ref{thm:TAZRP-symmetry}. By taking the same limit,
we can consider pathwise events for the mTAZRP in equilibrium. 
For example, Theorem \ref{thm:TAZRP-symmetry} tells
us that for the system in equilibrium, the probability 
that site $1$ is empty at time $0$ is symmetric in 
$x_2,\dots, x_n$; for $t=0$, we now have that the same is
true for the event that site $1$ is empty throughout
any given time interval $[0,T]$.

When $t=0$, we can think of site $i$ of the TASEP
as an ``exponential queueing server'', 
with service priority according to particle class.
It has some arrival process (formed of particles moving
from site $i-1$). It offers service at times
of a Poisson process of rate $x_i^{-1}$. Every time a service
occurs, if the queue is not empty, a particle of the highest 
type present is ``served'' and moves on to site $i+1$
(cyclically mod $n$). The departure process of site $i$
becomes the ``arrival process'' of site $i+1$, and so on. 
The Poisson processes at different sites
are independent. 

An \textit{interchangeability} result for exponential 
queueing servers in series was first proved for single-type systems
by Weber \cite{Weber1979}. Consider a system consisting of two 
exponential queueing servers in series, with the departure process
of the first becoming the arrival process of the second. 
Fix any arrival process to the first server.
Consider two versions of the system, one in which the service 
rates at queues $1$ and $2$ are $\mu_1$ and $\mu_2$ respectively,
and another in which the rates are swapped. 
The interchangeability result says that the distribution
of the departure process from the second server is 
the same in both versions of the system. 

By induction, the interchangeability result
is easily extended to permuting the rates 
of $k\geq 2$ independent exponential queueing servers in series. 

Using a coupling approach originating with Tsoucas and Walrand
\cite{TsoucasWalrand1987}, Martin and Prabhakar 
\cite{MartinPrabhakar2010}
extended 
the result to multi-class queues with priority. These are queues
which behave like the sites of the mTAZRP as described just above. 
For a system of queues in series with arrival process $A$ at the first 
queue and service processes $S_1, \dots, S_k$, we write 
$D(A, S_1, S_2, \dots, S_k)$ for the departure process from 
the final queue. Section 2.6 and Section 6 of \cite{MartinPrabhakar2010} give the following result:
\begin{prop}\label{prop:interchange}
Let $\mu_1, \mu_2, \dots, \mu_k>0$. Let $(\sigma_1, \dots, \sigma_k)$ 
be any permutation of $[k]$. 

There is a coupling of:
\begin{itemize}
\item[(a)]
independent Poisson processes $S_1, S_2, \dots, S_k$
with rates $\mu_1, \mu_2, \dots, \mu_k$ respectively, and
\item[(b)]
independent Poisson processes $\tS_1, \tS_2, \dots, \tS_k$
with rates $\mu_{\sigma_1}, \mu_{\sigma_2},\dots, \mu_{\sigma_k}$
respectively,
\end{itemize}
such that for all arrival processes $A$,
\[
D(A, S_1, S_2, \dots, S_k)
=
D(A, S_{\sigma_1}, S_{\sigma_2},\dots, S_{\sigma_k}).
\]
\end{prop}

Proposition \ref{prop:interchange} can be used
to construct a coupling which establishes Theorem 
\ref{thm:stronger-symmetry}. We outline the idea, which involves
coupling two mTAZRP systems. Sites $1,\dots, \ell$ 
have parameters $x_1,\dots, x_\ell$ in both systems. 
The remaining sites have parameters $x_{\ell+1},\dots, x_n$ 
in the first system, and $\tx_{\ell+1}, \dots, \tx_n$ in the second system, 
where $(\tx_{\ell+1},\dots, \tx_n)$ is some permutation
of $(x_{\ell+1},\dots, x_n)$. 

Both systems have the same initial configuration at time $0$. 
We can then construct them via the service processes at each site. 
We will use the same service processes in both systems at each of the 
sites $1,\dots,\ell$ (which are Poisson processes of rates $x_1^{-1}, \dots, x_{\ell}^{-1}$). Meanwhile we use 
Proposition \ref{prop:interchange} to couple 
the service processes at sites $\ell+1,\dots,n$
between the two systems, to give Poisson processes of rates
$x_{\ell+1}^{-1},\dots, x_n^{-1}$ in the first system 
and of rates $\tx_{\ell+1}^{-1},\dots, \tx_n^{-1}$ in the second system).

In the initial configuration, sites $\ell+1, \dots, n$ 
are empty. The first point in the ``arrival process'', 
namely the process of jumps from site $\ell$ to site $\ell+1$, 
occurs at the same time in both systems, because the 
service processes at sites $1,\dots, \ell$ are identical 
between the two systems. From then on that same coupling,
together with the coupling between the service processes 
at sites $\ell+1, \dots, n$ provided by Proposition \ref{prop:interchange}, ensures that the arrival process
and the ``departure process'' (the process of jumps from 
site $n$ to site $1$) stay identical between the two
systems. The configuration restricted to the sites $1,\dots,\ell$
is always identical between the two systems 
(although of course the configuration on sites $\ell+1, \dots, n$
may differ). 

We believe that the equivalent of Proposition \ref{prop:interchange}
also holds in the case $t>0$. In this case, at each queue $i$,
a service of level $a$ would occur at rate $\mu_i t^{a-1}$; 
when such a service occurs, the $a$th highest-priority customer
present departs (unless there are fewer than $a$ customers present,
in which case nothing happens). If such a generalisation 
of Proposition \ref{prop:interchange} holds, 
then the stronger symmetry property of 
Theorem \ref{thm:stronger-symmetry} could also be extended
to $t>0$.

\subsection{Densities and currents}
\label{sec:dens-curr}

In this section, we will compute densities and currents in the mTAZRP.
The main idea here is to derive the results for the single-species TAZRP and use the coloring idea stated formally in \cref{prop:colouring}.

Consider the mTAZRP with particle content given by $\lambda$ and with $n$ sites. We now assume without loss of generality that $\lambda$ is compressed, so that it is of the form $ \langle 1^{m_1}, 2^{m_2}, \dots, k^{m_k} \rangle$ with all $m_i$'s positive.

Let $\tau^{(j)}_i$ be the random variable counting the number of particles of type $j$ at site $i$ in a configuration of the mTAZRP. By convention, we will denote the expectation in the stationary distribution by $\langle \cdot \rangle$. 
The density of a given species (say $j$) at a given site (say $i$) is defined to be the expected number of particles of type $j$ at site $i$ in the stationary distribution and is thus given by $\langle \tau^{(j)}_i \rangle$. 

Note the following cyclic symmetry property, which is immediate from the definition of the process. 
\begin{prop}
\label{prop:trans-cov}
The mTAZRP with partition $\lambda$ on $n$ sites is invariant under simultaneous translation of the sites $i \to i+1$ for $1 \leq i \leq n-1$ and $n \to 1$, and relabelling of the site-dependent parameters $x_i \to x_{i+1}$ for $1 \leq i \leq n-1$ and $x_n \to x_1$.
\end{prop}
As an immediate consequence, we can deduce the following result for the densities.
\begin{prop}
\label{prop:dens-transinv}
Suppose $\langle \tau^{(j)}_1 \rangle = r(x_1,\dots,x_n)$. Then for any $i$, 
$\langle \tau^{(j)}_i \rangle = r(x_i,\dots,x_n,x_1, \allowbreak \dots,x_{i-1})$.
\end{prop}

By \cref{prop:dens-transinv}, it suffices to compute the densities of all species of particles at the first site. 
It turns out that there is a uniform formula. 
We first focus on the case of single species, defined in \cref{sec:single}. Consider the TAZRP with $m$ particles on a ring of size $n$.

\begin{prop}
\label{prop:dens}
The density at the first site is given by
\[
\langle \tau^{(1)}_1 \rangle = 
\frac{x_1 \partial_{x_1} 
\widetilde{H}_{\langle 1^{m} \rangle}(x_1,\dots,x_n; 1, t)}
{\widetilde{H}_{\langle 1^{m} \rangle}(x_1,\dots,x_n; 1, t)}
= x_1 \partial_{x_1} \log \widetilde{H}_{\langle 1^{m} \rangle}(x_1,\dots,x_n; 1, t).
\]
\end{prop}

\begin{proof}
By definition of the density, 
\begin{align*}
\langle \tau^{(1)}_1 \rangle = & \sum_\tau \tau_1 \pi(\tau) \\
=& \frac{1}{\widetilde{H}_{\langle 1^{m} \rangle}(x_1,\dots,x_n; 1, t)}
\sum_\tau \tau_1 \qbinom m{\tau_1,\dots, \tau_n}_t\ \prod_{i=1}^n x_i^{\tau_i} \\
=& \frac{1}{\widetilde{H}_{\langle 1^{m} \rangle}(x_1,\dots,x_n; 1, t)}
\sum_\tau \qbinom m{\tau_1,\dots, \tau_n}_t\ x_1 
\partial_{x_1} \left( \prod_{i=1}^n x_i^{\tau_i} \right),
\end{align*}
where we have used \cref{prop:single-species-prob} in the second line.
The linearity of the derivative completes the proof. 
\end{proof}

\begin{thm}
\label{thm:dens}
For $1 \leq j \leq k$, the density of the $j$'th species at the first site is given by
\[
\langle \tau^{(j)}_1 \rangle 
= x_1 \partial_{x_1} \log \left( 
\frac{\widetilde{H}_{\langle 1^{m_j + \cdots + m_k} \rangle}(x_1,\dots,x_n; 1, t)}
{\widetilde{H}_{\langle 1^{m_{j+1} + \cdots + m_k} \rangle}(x_1,\dots,x_n; 1, t)}
\right),
\]
where the denominator inside the logarithm is $1$ if $j = k$.
\end{thm}

\begin{proof}
The proof uses the coloring argument from \cref{prop:colouring}. First, we can construct a projection involving particles of species $j$ through $k$ to obtain a single-species TAZRP with $m_j + \cdots + m_k$ particles. From \cref{prop:dens}, the sum of densities,
$\langle \tau^{(j)}_1 + \cdots + \tau^{(k)}_1 \rangle$ is given by
$x_1 \partial_{x_1} \log \widetilde{H}_{\langle 1^{m_j + \cdots + m_k} \rangle}(x_1,\dots,x_n; 1, t)$.
Since we are only interested in species $j$, we can subtract off 
$\langle \tau^{(j+1)}_1 + \cdots + \tau^{(k)}_1 \rangle$ using the same argument to obtain our formula.
\end{proof}

Since the density is expressed as a Macdonald polynomial, we immediately see that it is symmetric in the variables $\{x_2,\ldots,x_n\}$. This also follows immediately from the $\ell=1$ case of 
\cref{thm:symm probs}.

Currents are important observables in statistical physics because they are a prime indicator of the lack of reversibility in the system.
The \emph{current} of a certain species of particle from site $i$ to site $j$ is defined as the number of particles of that species
jumping from $i$ to $j$ per unit of time in the large-time limit,
or in the stationary state. The current from $i$ to $i+1$ is the same for all sites $i$ because of particle conservation. We will derive explicit formulas for the current of all species of particles.

We begin with the single-species case. The current is then given by the  following result.

\begin{prop}
\label{prop:curr}
For the single-species TAZRP on $n$ sites with $m$ particles, the current is
\[
J = [m]_t \frac{\widetilde{H}_{\langle 1^{m-1} \rangle}(x_1,\dots,x_n; 1, t)}{\widetilde{H}_{\langle 1^{m} \rangle}(x_1,\dots,x_n; 1, t)}.
\]
\end{prop}

\begin{proof}
For concreteness we will focus on the jumping of particles from site $1$ to site $2$. Recall that states of the single species TAZRP can be represented by $\tau = (\tau_1,\dots,\tau_n) \models m$, weak compositions of $m$ with $n$ parts. From the definition of the current,
we have that
\[
J = \sum_{\substack{\tau \\ \tau_1 > 0 \\ \tau_1+\cdots+\tau_n=m}} \frac{[\tau_1]_t}{x_1} \pi(\tau),
\]
where the sum runs over all configurations of the $m$ particles amongst the $n$ sites such that the first site is not empty. The formula for the stationary probability $\pi(\tau)$ is given in \cref{prop:single-species-prob}.
We refine the above sum according to the number of particles at site $1$,
and plug in the formula for $\pi(\tau)$ to get
\begin{align*}
J =& \sum_{j=1}^m \frac{[j]_t \, x_1^{j-1}}{\widetilde{H}_{\langle 1^{m} \rangle}(x_1,\dots,x_n; 1, t)} 
\sum_{\substack{\tau_2,\ldots,\tau_n\geq 0 \\ j+\tau_2+\cdots+\tau_n=m}} 
\qbinom m{j, \tau_2,\dots, \tau_n}_t \, \prod_{i=2}^n x_i^{\tau_i} \\
=& \sum_{j=1}^m \frac{[j]_t \, x_1^{j-1}}{\widetilde{H}_{\langle 1^{m} \rangle}(x_1,\dots,x_n; 1, t)} \qbinom {m}{j}_t
\sum_{\substack{\tau_2,\ldots,\tau_n\geq 0\\\tau_2+\cdots+\tau_n=m-j }} 
\qbinom {m-j}{\tau_2,\dots, \tau_n}_t \, \prod_{i=2}^n x_i^{\tau_i}\\
=& \frac{[m]_t}{\widetilde{H}_{\langle 1^{m} \rangle}(x_1,\dots,x_n; 1, t)} 
\sum_{j=1}^m \qbinom {m-1}{j-1}_t x_1^{j-1}
\sum_{\substack{\tau_2,\ldots,\tau_n\geq 0\\ \tau_2+\cdots+\tau_n=m-j}} 
\qbinom {m-j}{\tau_2,\dots, \tau_n}_t \, \prod_{i=2}^n x_i^{\tau_i}.
\end{align*}
Now, using the recursion \eqref{eq:expand} for the modified Macdonald polynomials at $q=1$ given in \cref{it:Ht-recursion}, we obtain our result. 
\end{proof}

We can now prove the general result for currents of all species. Let
$\lambda = \langle 1^{m_1},\dots,k^{m_k} \rangle$.

\begin{thm}
\label{thm:curr}
Consider the multispecies TAZRP with content $\lambda$ on $n$ sites. Then,
for $1 \leq j \leq k$, the current of particle of species $j$ is given by
\[
[m_j + \cdots + m_k]_t \frac{ \widetilde{H}_{\langle 1^{m_j + \cdots + m_k - 1} \rangle}}
{\widetilde{H}_{\langle 1^{m_j + \cdots + m_k}\rangle}}
- [m_{j+1} + \cdots + m_k]_t 
\frac{\widetilde{H}_{\langle 1^{m_{j+1} + \cdots + m_k-1}\rangle}}
{\widetilde{H}_{\langle 1^{m_{j+1} + \cdots + m_k}\rangle}},
\]
where the modified Macdonald polynomials are evaluated at the variables $x_1,\dots,x_n$ with $q=1$.
\end{thm}

\begin{proof}
This is again a straightforward application of the colouring argument in \cref{prop:colouring}. We first obtain the total current of particles $j$ through $k$ by treating these as $1$'s using \cref{prop:curr}, and subtract the current of particles $j+1$ through $k$ in the same way.
\end{proof}

\subsection{The distribution at the top of the diagram}
\label{sec:top}

In this section, we consider a consistency property 
between tableaux of different shapes. 
We will show that under the stationary distribution
of the tableau process, the distribution of entries in
a given set of rows at the top of the tableau does 
not depend on the shape of the rows below. 

Suppose $\lambda$ has $\ell$ parts and $\mu$ is formed from $\lambda$ by deleting the bottom $k$ rows 
of $\lambda$. That is, $\mu = ((\lambda_1 - k)_+, \dots, (\lambda_\ell - k)_+)$. 
Given a filling $\xi\in\Tab(\lambda,n)$, there is a natural way to  derive from it a filling $\sigma=\xi|_\mu\in\Tab(\mu,n)$ by ignoring the $k$ bottom rows; that is, by setting $\xi|_\mu(r,c)=\xi(r+k, c)$.

\begin{thm}
If $\xi$ is distributed according to the stationary 
distribution of the tableau Markov chain on $\Tab(\lambda,n)$,
and $\sigma=\xi|_\mu$, then $\sigma$
is distributed according to the stationary 
distribution of the tableau Markov chain on $\Tab(\mu,n)$. 
\end{thm}

\begin{example}
Let $\lambda = (5,4,2,1,1)$ and $k = 2$ so that $\mu = (3,2)$ and let
\[
\sigma = 
\tableau{
1 \\
2 & 5 \\
3 & 2}
\quad \in \Tab(\mu,5).
\]
Then $\sigma=\xi|_\mu$ for any filling $\xi\in\Tab(\lambda,5)$ 
which has the pattern
\[
\tableau{
1 \\
2 & 5 \\
3 & 2 \\
\cdot &\cdot &\cdot \\
\cdot &\cdot &\cdot &\cdot &\cdot }
\]
\end{example}

\begin{proof}
It suffices to prove this for $k = 1$, i.e.~deleting the bottom row in the diagram of $\lambda$. Since $\lambda$ has length $\ell$, this amounts to proving
\begin{equation}
\label{wt-sum}
\sum_{w \in [n]^\ell} \wt \begin{pmatrix}
\sigma \\
w
\end{pmatrix}
=
\tilde{H}_{1^\ell}(x_1,\dots,x_n; 1, t) \wt(\sigma),
\end{equation}
where we have appealed to \eqref{eq:Ht-factor} for the right hand side.
Note that the weight on the left hand side depends on the content of the tableau and the quinv statistic, which only affects two consecutive rows. Thus $w$ affects the weight of the composite tableaux only through the bottom row of $\sigma$.
Thus, it suffices to restrict attention to two-row tableaux. The general structure of a two-row tableau is
\[
\ytableausetup{boxsize=2.6em}
\begin{ytableau}
t_1 & \dots & t_m \\
w_1 & \dots & w_m & w_{m+1} & \dots & w_\ell
\end{ytableau}
\]
where the bottom row has length $\ell$ and the top row has length $m$, say.
We will now prove \eqref{wt-sum} for such two-row tableaux using induction on $m$. Before proving this, we note that we can further refine the sum in \eqref{wt-sum} according to the content of $w$. If $w$ has $m_1$ $1$'s, $m_2$ $2$'s, and so on, then the weight of $w$ is proportional to $x_1^{m_1} x_2^{m_2} \cdots x_n^{m_n}$. Thus, we need to consider only the contribution of the quinv statistic. From the structure of the the modified Macdonald polynomial in 
\eqref{eq:Ht-monomialsum}, we have to show that this is the $t$-multinomial coefficient $\begin{bmatrix}
\ell \\ m_1,\dots,m_n \end{bmatrix}_t$, which is symmetric under any permutation of the $m_i$'s.

The base case of the induction is $m = 0$. In that case, the only contribution to the quinv statistic is from degenerate pairs $(w_i, w_j)$, where $i < j$ and $w_i < w_j$. This is the same as the number of coinversions and it is well-known that the sum of $t^{\coinv(w)}$ gives the desired $t$-multinomial coefficient.

We now suppose that $m, \ell > 0$ and we assume the result to hold when there are $m-1$ rows on top and $\ell - 1$ rows at the bottom. We then add extra boxes to the left on both rows. Suppose the top-left box is occupied by $a$ and the bottom-left, by $k$. Then the extra contribution to the quinv statistic from these two boxes depends on the relative order of $a$ and $k$ and there are three cases:

\begin{enumerate}
\item $k < a$: all elements between $k$ and $a$ contribute.

\item $k = a$: all elements except $a$ contribute.

\item $k > a$:  all elements larger than $k$ or smaller than $a$ contribute.
\end{enumerate}

Recall that we have supposed the content of $w$ to be $(m_1,\dots,m_n)$. Using the induction hypothesis, the total contribution is
\begin{multline}
\sum_{k=1}^{a-1} t^{m_{k+1} + \cdots + m_{a-1}} 
\qbinom{\ell-1}{m_1,\dots,m_k-1,\dots,m_n}_t 
+ t^{\ell - m_a} \qbinom{\ell - 1}{m_1,\dots,m_a-1,\dots,m_n}_t \\
+ \sum_{k=a+1}^n t^{m_1 + \cdots + m_{a-1} + m_{k+1} + \cdots + m_n}
\qbinom{\ell - 1}{m_1,\dots,m_k-1,\dots,m_n}_t.
\end{multline}
Keeping only relevant factors inside the sum, we get
\begin{multline}
\label{qbinom-sum}
\qbinom{\ell - 1}{m_a + \cdots + m_n}_t 
\qbinom{m_a + \cdots + m_n}{m_a,\dots, m_n}_t 
\sum_{k=1}^{a-1} t^{m_{k+1} + \cdots + m_{a-1}} \qbinom{m_1 + \cdots + m_{a-1} - 1}{m_1,\dots,m_k-1,\dots,m_{a-1}}_t \\
+ t^{\ell - m_a} \qbinom{\ell - 1}{m_1,\dots,m_a-1,\dots,m_n}_t 
+ t^{m_1 + \cdots + m_{a-1}} \qbinom{\ell - 1}{m_1 + \cdots + m_a}_t \\
\times \qbinom{m_1 + \cdots + m_a}{m_1,\dots, m_a}_t
\sum_{k=a+1}^n t^{m_{k+1} + \cdots + m_n}
\qbinom{m_{a+1} + \cdots + m_n - 1}{m_{a+1},\dots,m_k-1,\dots,m_n}_t.
\end{multline}

We now recall the fundamental recursion for the $t$-multinomial coefficients:
\begin{equation}
\label{qbinom-recur}
\qbinom{\ell}{m_1,\dots,m_n}_t 
= \sum_{k=1}^n t^{m_{k+1} + \cdots + m_n}
\qbinom{\ell - 1}{m_1,\dots,m_{k}-1,\dots,m_n}.
\end{equation}
Using \eqref{qbinom-recur}, the first and last sums in \eqref{qbinom-sum}
can be evaluated to obtain
\begin{multline}
\qbinom{\ell - 1}{m_a + \cdots + m_n}_t 
\qbinom{m_a + \cdots + m_n}{m_a,\dots, m_n}_t 
\qbinom{m_1 + \cdots + m_{a-1}}{m_1,\dots,m_{a-1}}_t \\
+ t^{\ell - m_a} \qbinom{\ell - 1}{m_1,\dots,m_a-1,\dots,m_n}_t \\
+ t^{m_1 + \cdots + m_{a-1}} \qbinom{\ell - 1}{m_1 + \cdots + m_a}_t 
\qbinom{m_1 + \cdots + m_a}{m_1,\dots, m_a}_t
\qbinom{m_{a+1} + \cdots + m_n}{m_{a+1},\dots,m_n}_t.
\end{multline}
Taking two $t$-multinomial coefficient factors common, we obtain
\begin{multline}
\qbinom{m_1 + \cdots + m_{a-1}}{m_1,\dots,m_{a-1}}_t
\qbinom{m_{a+1} + \cdots + m_n}{m_{a+1},\dots,m_n}_t 
\Bigg( \qbinom{\ell - 1}{m_a,m_{a+1} + \cdots + m_n, m_1 + \cdots + m_{a-1}-1}_t  \\
+ t^{\ell - m_a} \qbinom{\ell - 1}{m_a-1,m_{a+1} + \cdots + m_n, m_1 + \cdots + m_{a-1}}_t \\
+ t^{m_1 + \cdots + m_{a-1}} \qbinom{\ell - 1}{m_a,m_{a+1} + \cdots + m_n-1, m_1 + \cdots + m_{a-1}}_t 
\Bigg).
\end{multline}
The sum of the three $t$-multinomial coefficients can again be performed using \eqref{qbinom-recur}, and we simplify to obtain
\[
\qbinom{\ell}{m_1,\dots,m_n}_t,
\]
completing the proof.
\end{proof}

\section{A multiline process in bijection with the tableau process}
\label{sec:multiline}
Our tableau process is very closely related to the 
``multiline diagrams'' (or ``multiline queues'') that have been used
to study the stationary distributions of several multi-species 
interacting particle systems 
(starting with the TASEP in \cite{FM07}), and to give formulas
for Macdonald polynomials in \cite{CMW18}.

Fix $n\in\NN$, and let $\lambda$ have largest part $\lambda_1=L$. For $1\leq k \leq L$, define $\lambda^{(k)}=(u \in \lambda \mid u\geq k)$ to be the restriction of $\lambda$ to parts of size at least $k$. 
For such a partition $\lambda$, a multiline diagram of type $(\lambda, n)$ consists of configurations on a lattice with $L$ rows and $n$ columns, such that row $k$ contains $\lambda'_k$ particles with the labels $\lambda^{(k)}$ for $1 \leq j \leq L$.
Each lattice site may contain any number of particles, with particles of the same type being indistinguishable. In particular, one may think of each individual row of the diagram as an element of the set of mTAZRP configurations $\TAZRP(\lambda^{(k)},n)$, defined in \cref{sec:tazrp}. All the labels of particles in row $k$ are
in $[k,L]$. The multiset of labels in row $k-1$ is obtained
by adding some particles of type $k-1$ to the multiset for row $k$. 

Represent a multiline diagram of type $\lambda,n$ by a configuration
\[
(M^{(L)},\ldots,M^{(2)},M^{(1)})\in \TAZRP(\lambda^{(L)},n)\times\cdots\times\TAZRP(\lambda^{(2)},n)\times\TAZRP(\lambda^{(1)},n).
\]
Then row $k$ of the diagram is given by the configuration $M^{(k)}\in \TAZRP(\lambda^{(k)},n)$. 
\begin{example}\label{ex:multiline}
Let $\lambda=(4,3,1)$ and fix $n=5$. Then $\lambda^{(4)}=(4)$, $\lambda^{(3)}=\lambda^{(2)}=(4,3)$, and $\lambda^{(1)}=(4,3,1)$. The diagram below is an example of a multiline diagram of type $(\lambda,n)$, where for visual convenience we pair the particles of the same labels between adjacent rows:
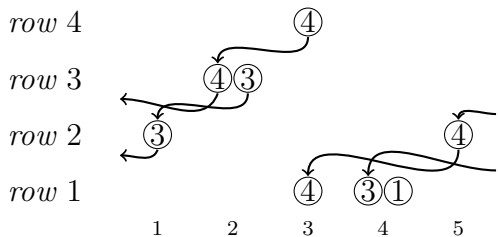
\begin{center}
\begin{tikzpicture}[scale=0.5]
\def \h {1.5}
\def \w {2}
\def \r {.37cm}
\def \rr {.4}

\foreach \i in {0,...,4}
{
\draw[gray!50] (0+.5*\w,\i*\h-.5*\h)--(\w*5+.5*\w,\i*\h-.5*\h);
}
\foreach \i in {1,...,4}
{
\node at (-1,\i*\h-\h) {\tiny \normalfont{row} $\i$};
}
\foreach \i in {0,...,5}
{
\draw[gray!50] (\w*\i+.5*\w,-.5*\h)--(\w*\i+.5*\w,3.5*\h);
}
\foreach \i in {1,...,5}
{
\node at (\i*\w,-1.5) {\tiny $\i$};
}

 \draw (3*\w,3*\h) circle (\r) node[color=black] {$4$};
\draw[thick,->,black] (3*\w, 3*\h-\rr)--(3*\w,2.5*\h)--(2*\w+\rr,2.5*\h)--(2*\w+\rr, 2*\h+\rr); 

\draw (2*\w-\rr,2*\h) circle (\r) node[color=black] {$3$};
\draw (2*\w+\rr, 2*\h) circle (\r) node[color=black] {$4$};
\draw[thick,->,black] (2*\w+\rr, 2*\h-\rr)--(2*\w+\rr,1.5*\h)--(.5*\w-.3, 1.5*\h); 
\draw[thick,->,black] (5.5*\w+.3, 1.5*\h)--(5*\w,1.5*\h)--(5*\w,\h+\rr); 
\draw[thick,->,black] (2*\w-\rr, 2*\h-\rr)--(2*\w-\rr,1.5*\h+.2)--(1*\w,1.5*\h+.2)--(1*\w,\h+\rr); 

\draw (5*\w,\h) circle (\r) node[color=black] {$4$};
\draw (1*\w,\h) circle (\r) node[color=black] {$3$};
\draw[thick,->,black] (1*\w, 1*\h-\rr)--(1*\w,.5*\h)--(.5*\w-.3, .5*\h); 
\draw[thick,->,black] (5.5*\w+.3, .5*\h)--(4*\w+\rr,.5*\h)--(4*\w+\rr,0*\h+\rr); 
\draw[thick,->,black] (5*\w, 1*\h-\rr)--(5*\w,.5*\h+.2)--(3*\w,.5*\h+.2)--(3*\w,0*\h+\rr); 

\draw (3*\w,0) circle (\r) node[color=black] {$4$};
\draw (4*\w-.4,0) circle (\r) node[color=black] {$1$};
\draw (4*\w+.4,0) circle (\r) node[color=black] {$3$};

\end{tikzpicture}
\end{center}
As a tuple of TAZRP states, the $x$-monomial is $x^M=x_1x_2^2x_3^2x_4^2x_5$, and the configuration is
\[
M=\big((\cdot|\cdot|4|\cdot|\cdot),\ (\cdot|43|\cdot|\cdot|\cdot),\ (3|\cdot|\cdot|\cdot|4),\ (\cdot|\cdot|4|31|\cdot)\big)
\]
\end{example}

We now describe a map $\cM$ from the set $\Tab(\lambda, n)$ of fillings to the set of multi-line diagrams of type $\lambda, n$. 
For $\sigma\in\Tab(\lambda,n)$, denote row $j$ of $\sigma$ by $\sigma_{(j)}$ for each $j$. Then define $\mathcal{M}(\sigma) = (M^{(L)},\ldots,M^{(2)},M^{(1)})$, where $M^{(j)} = \proj(\sigma_{(j)})$, $\proj$ being the function mapping fillings to states of the TAZRP in \cref{def:f}.

From here on we will focus on the case where $\lambda$ is a 
strict partition, i.e.\ where the mTAZRP has at most one 
particle of any given species. (See the end of the section 
for comments on the more general case.) 

When $\lambda$
is a strict partition, the map $\cM$ is a bijection, 
and we can describe its inverse as follows. 
For a configuration $w^{(k)}\in\TAZRP(\lambda^{(k)}, n)$
which contains a particle of type $r$, 
denote by $p_r(w^{(k)})$ the position of this particle. 
Now let 
$M=(M^{(L)},\ldots,M^{(2)},M^{(1)})$ be a multiline diagram of type $(\lambda,n)$. The filling $\sigma=\cM^{-1}(M)\in\Tab(\lambda,n)$ 
can then be described in the following way. For each cell $(k,j)\in\dg(\lambda)$, set $\sigma(k,j)=p_{\lambda_j}(M^{(k)})$. In other words, the content of $\sigma$ in cell $(k,j)$  
records the position of the (unique) particle with label $\lambda_j$ in $M^{(k)}$. In this way, the positions of the particles in the (unique) string of type $\lambda_j$ in $M$ are recorded in the $j$'th column of $\sigma$.

\begin{example}\label{ex:Msigma}
If $M$ is the multiline diagram in \cref{ex:multiline},
then the unique tableau $\sigma$ with $M=\cM(\sigma)$ is 
\[\sigma=\raisebox{14pt}{\tableau{3\\2&2\\5&1\\3&4&4}}
\]. 
\end{example}

Now we define the weight $\wt(M)$ of a multiline diagram $M$, which is a monomial in $x_1,\ldots,x_n$ and $t$. The $x$-monomial, denoted by $x^M$, is given by the product
$\displaystyle x^M=\prod_{j=1}^n x_j^{c_j(M)}$, where $c_j(M)$ is the total number of particles in column $j$ of $M$ across all the rows.

To get the power of $t$, we will define the notion of 
a \textit{refusal} in the multiline diagram, motivated by
a queueing interpretation that we explain below. 
First, some useful notation: when the numbers $1$ through $n$ are arranged clockwise on a circle, for $a,b,c \in [n]$, we say that $a<b<c$ if $b$ is strictly between $a$ and $c$ when reading clockwise, including the case when $a=c\neq b$. 
Let $\bI[f]$ be the 
\emph{indicator function} of $f$,
which equals $1$ if $f$ is true and $0$ otherwise. For a configuration $w^{(k)}\in\TAZRP({\lambda}^{(k)},n)$, denote by $p_r(w^{(k)})$ the position of the particle with type $r$ in $w^{(k)}$. 

Now, consider a pair of rows $k$ and $k-1$ in $M$. For the configuration $(M^{(k)}, M^{(k-1)})$ in those two rows, the \emph{number of refusals} $R(M^{(k)}, M^{(k-1)})$ between $M^{(k)}$ and $M^{(k-1)}$ is defined as follows. Let
\begin{equation}\label{Rdef}
R(M^{(k)}, M^{(k-1)})=
\sum_{k-1\leq s<r\leq L}
\bI\big[p_r(M^{(k-1)})<p_s(M^{(k-1)})<p_r(M^{(k)})\big],
\end{equation}
(where the sum 
is over $(r,s)$ such that $r$ appears as a label
in row $k$ and $s$ appears as a label in row $k-1$). 
Finally, define
\[
\wt(M) =  x^m \,t^{\sum_{k=2}^L R(M^{(k)},M^{(k-1)})}.
\]

\begin{remark}
We can think of the two-line sub-diagram of a multiline diagram $M$ restricted to the rows $k,k-1$ as a \emph{queue} consisting of an arrival process $M^{(k)}$ and a departure process $M^{(k-1)}$. 
In this way the entire multiline diagram consists of a 
system of $L-1$ queues in series. ``Time'' is indexed by $[n]$, is cyclic, and moves from right to left. The queueing 
discipline gives particles with higher labels priority over
those with lower labels. 

For the queue consisting of rows $k$ and $k-1$, 
the arrival time of a particle $r\in[k,L]$ 
is its position $p_r(M^{(k)})$ in row $k$, and its departure time 
is its position $p_r(M^{(k-1)})$ in row $k-1$. 
The position $p_{k-1}(M^{(k-1)})$ of a particle of type $k-1$ in 
row $k-1$ is an ``unused service time". 
Our convention is that if the particle arrives at a position $j$, the first time it can depart is $j-1$. The strings connecting the particles between adjacent rows (as shown in Example 
\ref{ex:multiline})
illustrate the corresponding arrival and service times of the particle with a particular label. 

Consider $c<b<a$. 
Suppose a particle of a given species arrives at $a$ and departs at $c$, but at $b$, a weaker species departs or an unused service time occurs.
Then the service time was ``available'' to the stronger particle and was ``refused'' by it. This corresponds precisely to
the pattern 
$p_r(M^{(k-1)}) < p_s(M^{(k-1)}) < p_r(M^{(k)}$
for $s<r$ counted by \eqref{Rdef}. Hence we call such a triple 
of positions a 
\emph{refusal} in \eqref{Rdef}. Each such refusal contributes
a factor of $t$ to the weight of the configuration. 
\end{remark}

\begin{example}\label{ex:refusals}
In \cref{ex:Msigma}, $x^M=x_1x_2^2x_3^2x_4^2x_5$.
We have $R(M^{(4)},M^{(3)}) = 0$ and $R(M^{(3)},M^{(2)}) = 1$ 
since particle $3$ are served instead of particle $4$
at time $1$, giving
$p_4(M^{(3)}) < p_3(M^{(3)}) < p_4(M^{(4)})$; 
and $R(M^{(2)},M^{(1)}) = 2$
since both particles $1$ and $3$ are served instead of particle $4$
at time $4$,
giving 
$p_4(M^{(1)}) < p_3(M^{(1)}) < p_4(M^{(2)})$ and 
$p_4(M^{(1)}) < p_1(M^{(1)}) < p_4(M^{(2)})$. This gives 
a total of $3$ refusals, and so $\wt(M)=x_1x_2^2x_3^2x_4^2x_5 t^3$.
\end{example}

\begin{lemma}
If $\sigma$ is a filling and $M=\cM(\sigma)$ is its corresponding
multiline diagram, then
\[\wt(M) =t^{\quinv(\sigma)}x^{\sigma}=\wt(\sigma).\]
\end{lemma}

\begin{proof}
This is straightforward to check. First, $x^{M} = x^{\sigma}$ by definition, since for each site $j$ and row $k$, $M_{j}^{(k)}$ precisely corresponds to the set of cells with content $j$ in row $k$ of $\sigma$.

Now, let $x=(k,i)$, $y=(k-1,i)$, and $z=(k-1,j)$ with $i<j$ be the cells of a triple in $\sigma$. Then $x$ and $y$ correspond to the particles with label $\lambda_i$ in row $k$ and $k-1$ of $M$, respectively, and $z$ corresponds to the particle with label $\lambda_j$ in row $k-1$, and $\lambda_j<\lambda_i$. $\sigma(x)=p_{\lambda_i}(M^{(k)})$, $\sigma(y)=p_{\lambda_i}(M^{(k-1)})$, and $\sigma(z)=p_{\lambda_j}(M^{(k)})$. Then $(\sigma(x),\sigma(y),\sigma(z))\in\cQ$ if and only if $p_{\lambda_i}(M^{(k-1)})<p_{\lambda_j}(M^{(k-1)})<p_{\lambda_i}(M^{(k)})$. Summing over all triples of $\sigma$, we get $\quinv(\sigma)=\sum_{2 \leq k \leq L} R(M^{(k)}, M^{(k-1)})$, thus matching the $t$-components of the weights as well. 
\end{proof}

\begin{example}
We continue with the tableau $\sigma$ from 
\cref{ex:Msigma}, and the corresponding multiline diagram
$M=\cM(\sigma)$ from \cref{ex:multiline}.

The $\quinv$ triples of $\sigma$ are $\qtrip{2}{5}{1}$ and two copies of $\qtrip{5}{3}{4}$, which precisely match the three refusals in $M$
described in \cref{ex:refusals}.
\end{example}

We can also give a natural description of the Markov chain on multiline
diagrams which is the image of the tableau Markov chain under the map $\cM$. (Since we remain in the case of a strict partition $\lambda$, 
this is the more simple version of the tableau chain described in Section \ref{sec:markov distinct}, which does not
involve the swapping operators $\{\tau_j\}$.)

In terms of multiline diagrams, the chain can be described as follows. 
From a given multiline diagram $M$, any given particle, say of type $r$
at site $j$ on row $k$, may initiate a jump by jumping from site $j$ to site $j+1$. Then recursively, whenever a particle of a given type $r$ jumps from a site $j$ to a site $j+1$ on row $k$, if there is also a particle of type $r$ at site $j+1$ on the row above, then that particle jumps from site $j+1$ to site $j+2$.  

The rate of such a jump initiated by the particle of type $r$ at site $j$ on row $k$ can be described as follows. If $k>1$, and there is
also a particle of type $r$ at site $j$ on row $k-1$, then the jump cannot happen. Otherwise, it happens at rate 
$x_j^{-1} t^d$ where $d$ is equal to the number of stronger particles at site $j$ on row $k$, plus the number of weaker particles 
at site $j$ on row $k-1$. (If $k=1$ then there is no row $k-1$ and 
the second contribution is understood to be $0$.)

This corresponds precisely to the rate of the jump given by
$R_{(k,j)}$ in the tableau process from the state $\sigma=\cM^{-1}(M)$,
as described in Section \ref{sec:markov distinct}.

Notice in particular that the projection onto the bottom row 
of the diagram is precisely the mTAZRP process on $\TAZRP(\lambda, n)$. 

Finally, we discuss the case of a general partition $\lambda$,
in which there may be several particles with the same label
on the same row. Along similar lines to the case of the 
ASEP in \cite{martin-2020}, one can quite naturally define
a weight function directly on multiline diagrams in such a
way that $\wt(M)=\sum_{\sigma:\cM(\sigma)=M} \wt(\sigma)$. 
There are then various natural ways to define
a Markov process on the set of multiline diagrams,
such that the bottom row performs the mTAZRP process, and
such that the stationary distribution is proportional 
to the weight function. However, we found it easier to work
with the process on tableaux defined in Section \ref{sec:Markov},
whose image via the map $\cM$ onto the space
of multiline diagrams is not necessarily 
a Markov chain. 

\begin{remark}
In retrospect, the fact that a zero range process is the right candidate to associate to $\widetilde{H}_{\lambda}$ should not have been surprising: a fundamental difference between ZRPs and exclusion processes is that the former are free of the exclusion restriction, which prevents multiple particles from occupying the same site, as in the ASEP. In a general sense, this ``removal of exclusivity'' can be considered an analog of the action of plethysm when $P_{\lambda}$ is transformed into $\widetilde{H}_{\lambda}$. We shy away from exploring this connection in this article, but it would be greatly illuminating to understand how formal plethystic substitution manifests in passing from the ASEP to the TAZRP.
\end{remark}


\section{Proofs}
\label{sec:proofs}

\subsection{Proof of \cref{lem:updown}}
First we fix some notation to facilitate our proofs. 
For $j\geq 0$ let $\cont_j(\sigma):=\cont_j(\sigma;k)=\{(j,c)\in\dg(\lambda):\sigma(j,c)=k\}$ be the set of cells in row $j$ of $\sigma$ with content $k$, and let $\ell_j:=\ell_j(k)=|\cont_j(\sigma)|$ be the number of such cells. Also define $d_j(t) := d_j^{(k)}(t)=\sum_{u\in \cont_j(\sigma)} t^{\down(\sigma,u)}$ and $u_j(t) := u_j^{(k)}(t)=\sum_{u\in \cont_j(\sigma)} t^{\up(\sigma,u)}$ to be the generating functions for the $\down$ and $\up$ statistics of cells with content $k$ in row $j$. By convention, let $\cont_j(k)=\emptyset$ and $d_j(t) = u_j(t) =0$ if $j<1$ or $j>\lambda_1$. Finally let
 \[
D(\sigma,k) 
= \sum_{j=1}^{\lambda_1} d_j(t)
\quad \mbox{and} \quad 
U(\sigma,k) = \sum_{j=1}^{\lambda_1} u_j(t). 
 \]

\begin{example}
\label{eg:updown}
For example, consider the tableau in \cref{fig:polyqueue}, focusing only on the cells with content $k=3$:
\[
\sigma\big\vert_3=\raisebox{7pt}{\tableau{3&\ \\\ &3&\ &3&3\\\ &\ &\ &3&\ &3&3}}.
\]
Then $\ell_1 = \ell_2 = 3$, $\ell_3 = 1$ and
\begin{align*}
&d_1(t)=1+t+t^2,& \quad&d_2(t)=2t^3+t^4,& \quad&d_3(t)=t^3,\\
&u_1(t)=2t^3+t^4,&& u_2(t)=t^3+t^2+t,&&u_3(t)=1,
\end{align*}
so that $D(\sigma,3) = U(\sigma,3)= 1+t+t^2+3t^3+t^4$.
\end{example} 

We begin with a useful lemma.

\begin{lemma}
\label{lem:diff}
Fix $\lambda,n,k$, and let $\cont_j(\sigma)$, $d_j(t)$, and $u_j(t)$ be defined as above for some filling $\sigma$ of $\dg(\lambda)$. Then for all $j\geq 0$, we have
\begin{equation}
d_{j+1}(t)-u_j(t) = \frac{t^{\ell_{j+1}}-t^{\ell_{j}}}{t-1}.
\end{equation}
\end{lemma}

\begin{proof}
We use induction on the number of columns $s=\ell(\lambda)$ in $\dg(\lambda)$. 
When $s=1$, $\down(\sigma,u)=\up(\sigma,u)=0$ for every cell $u\in\dg(\lambda)$ regardless of the filling $\sigma$, and the cardinality of $\cont_j(\sigma)$ is either 0 or 1. The result then follows by a simple case analysis. 

Now suppose the statement holds for all fillings with fewer than $s$ columns. Fix $j\geq 0$ to be a row number. Suppose $\lambda=(\lambda_1,\ldots,\lambda_s)$, and let $\lambda'=(\lambda_2,\ldots,\lambda_s)$ be $\lambda$ with its leftmost column removed. Let $\sigma$ be a filling of $\dg(\lambda)$ and let $\sigma'$ be the sub-filling of $\sigma$ restricted to $\dg(\lambda')$. Define $d_j(t)$ and $u_j(t)$ to be the $\down$ and $\up$ generating polynomials corresponding to $\sigma$, and $d'_j(t)$ and $u_j'(t)$ those corresponding to $\sigma'$. Let $\ell_j$ and $\ell'_j$ be the cardinalities of $\cont_j(\sigma)$ and $\cont_j(\sigma')$, respectively. Since $\lambda'$ has fewer than $s$ columns, $d'_{j+1}(t)-u'_j(t) = \frac{t^{\ell'_{j+1}}-t^{\ell'_{j}}}{t-1}$.

Define $v_j=\delta_{\sigma(j,1),k}$ be the indicator function that equals 1 if the entry in the leftmost column of row $j$ of $\sigma$ has content $k$. Then $
d_{j}(t)=v_{j} t^{\ell'_{j-1}} + t^{v_j} d'_{j}(t)$,
since the lower arm of the cell $(j,1)$ consists of all cells in $\cont_{j-1}(\sigma')$, and every cell in $\cont_j(\sigma')$ has the cell $(j,1)$ in its lower arm if $v_j=1$. 
Similarly, $u_{j}(t) = v_j t^{\ell'_j} + t^{v_{j+1}}u'_j(t)$,
since the upper arm of the cell $(j,1)$ consists of all cells in $\cont_j(\sigma')$, and every cell in $\cont_j(\sigma')$ has the cell $(j+1,1)$ in its upper arm if $v_{j+1}=1$. 
Thus we have 
\begin{align*}
d_{j+1}(t)-u_{j}(t)&=v_{j+1} t^{\ell'_{j}} + t^{v_{j+1}} d'_{j+1}(t)-(v_j t^{\ell'_j} + t^{v_{j+1}}u'_j(t))\\
&=t^{v_{j+1}}(d'_{j+1}(t)-u'_j(t))+ t^{\ell'_{j}}(v_{j+1}-v_j)\\
&=t^{v_{j+1}}\frac{t^{\ell'_{j+1}}-t^{\ell'_{j}}}{t-1}+ t^{\ell'_{j}}(v_{j+1}-v_j)
\end{align*}
where the last equality follows by the induction hypothesis.

Set $\ell_j'$ and $\ell_j$ to be the number of cells with content $k$ in row $j$ of $\sigma'$ and $\sigma$, respectively. As such, $\ell_j=\ell'_j+v_j$. 
We can now rewrite the right hand side of the last line as
\[
\frac{t^{\ell_{j+1}}-t^{\ell'_j}(t^{v_{j+1}}+(v_{j}-v_{j+1})(t-1)) }{t-1}.
\]
Now,
\[
(t^{v_{j+1}}+(v_{j}-v_{j+1})(t-1)) \quad =\quad  \begin{cases} t, & v_j=v_{j+1}=1,\\
1+(t-1), & v_j=1, v_{j+1}=0,\\
t-(t-1), & v_j=0, v_{j+1}=1,\\
1, & v_j=v_{j+1}=0.
\end{cases}
\quad = \quad  t^{v_j}.
\]
Thus the expression simplifies to
\[
d_{j+1}(t)-u_j(t) = \frac{t^{\ell_{j+1}}-t^{\ell_{j}}}{t-1},
\]
which is precisely what we wanted to prove.
\end{proof}

For the tableau in \cref{eg:updown}, observe that 
\[
d_1(t)-u_0(t)=1+t+t^2=\frac{t^3-1}{t-1}=\frac{t^{\ell_1}-t^{\ell_0}}{t-1},\quad d_2(t)-u_1(t)=0=\frac{t^{\ell_2}-t^{\ell_1}}{t-1},
\]
\[
d_3(t)-u_2(t)=-t^2-t=\frac{t-t^3}{t-1}=\frac{t^{\ell_3}-t^{\ell_2}}{t-1},\quad 
d_4(t)-u_3(t)=-1=\frac{1-t}{t-1}=\frac{t^{\ell_4}-t^{\ell_3}}{t-1},
\]
in accordance with \cref{lem:diff}.
With \cref{lem:diff} in hand, we are prepared to prove \cref{lem:updown}.

\begin{proof}[Proof of \cref{lem:updown}]
We will first prove that $U(\sigma,k) = D(\sigma,k)$.
We begin by rearranging the entries in the sum:
\begin{align*}
\sum_{j=1}^{\lambda_1} (d_j(t)-u_j(t)) & =  d_1(t)+\sum_{j=1}^{\lambda_1-1} \big(d_{j+1}(t)-u_j(t)\big) -u_{\lambda_1}(t)\\
&=d_1(t)+\sum_{j=1}^{\lambda_1-1}\frac{t^{\ell_{j+1}}-t^{\ell_{j}}}{t-1} - u_{\lambda_1}(t)\\
&=d_1(t)+\frac{t^{\ell_{\lambda_1}} - t^{\ell_1}}{t-1} - u_{\lambda_1}(t).
\end{align*}
The second line above comes from \cref{lem:diff}, and the third line comes from simplifying the telescoping sum. Now we observe that 
\[
d_1(t) = 1+t+\ldots t^{\ell_1-1} = \frac{t^{\ell_1}-1}{t-1}
\]
since the only contribution to the leg of a cell in the bottom-most row is from the cells to its left in the same row.
Similarly,
\[
u_{\lambda_1}(t) = 1+t+\ldots t^{\ell_{\lambda_1}-1} = \frac{t^{\ell_{\lambda_1}}-1}{t-1}.
\]
since the only contribution to the upper leg of a cell in the top-most row is from the cells to its right in the same row. 
It follows that all terms in the equation above cancel, and we get
\[
\sum_{j=1}^{\lambda_1} (d_j(t)-u_j(t)) =0,
\]
which proves that $U(\sigma,k) = D(\sigma,k)$.

To now prove \eqref{eq:D=U}, it suffices to observe that if a cell $\sigma(u)=\sigma(\South(u))=k$, then $\down(\sigma,u)=\up(\sigma,\South(u))$. Thus if $\sigma(u)=\sigma(\South(u))=k$, the contribution of $u$ to $D(\sigma,k)$ cancels with the contribution of $\South(u)$ to $U(\sigma,u)$. Equivalently, if $\sigma(u)=\sigma(\North(u))=k$, then the contribution of $u$ to $U(\sigma,k)$ cancels with the contribution of $\North(u)$ to $D(\sigma,k)$. 
This proves the result.
\end{proof}

\subsection{Proof of \cref{lem:quinvdiff}}

\begin{proof}
Let $u=(i,j)$, $k:= \xi(u)$, and $y= H_u(\xi)=(i+m,j)$, with $\xi(y)=k+m$. By \cref{lem:g}, $y$ and $m$ are well-defined. Note that by definition, $m$ is maximal such that $\xi(i+\ell,j)=k+\ell \pmod n$ for all $0 \leq \ell \leq m$. In \cref{fig:xy}, we show the cells whose content differs between $\xi$ and $\sigma$, which is the contiguous increasing chain from $u$ to $y$. In particular, $\xi(i+\ell,j)=k+\ell$ and $\sigma(i+\ell,j)=k+\ell+1$ for all $0 \leq \ell \leq m$, and for all other cells $(i,j)\in\dg(\lambda)$, $\xi(i,j)=\sigma(i,j)$.

\begin{figure}[h!]

\begin{align*}
& \quad\ \xi & & \quad\ \sigma & \\
\begin{ytableau}
\none[{\scriptstyle y=(i+m,j): }] \\
\none[\vdots] \\
\none[\vdots] \\
\none[{\scriptstyle u=(i,j):}]
\end{ytableau}
\qquad &
\ytableaushort{{\scriptstyle k+m}, \vdots, \vdots, {\scriptstyle k}} 
&
 \raisebox{-1.2cm}{$\longrightarrow$} \qquad\quad
\qquad &
\ytableaushort{{\scriptstyle k+m+1}, \vdots, \vdots, {\scriptstyle k+1}} &
\end{align*}

\caption{The cells $\{u=(i,j),(i+1,j),\ldots,y=(i+m,j)\}$ are shown in $\dg(\lambda)$ with their respective contents in $\xi$ and $\sigma$. These cells are precisely those whose content is incremented by 1 via the ringing path transition triggered by $u$ in $\xi$ to obtain $\sigma=R_u(\xi)$.}\label{fig:xy}
\end{figure}

We will compute $\quinv(\xi)-\quinv(\sigma)$ by examining the contribution to quinv from triples coming from the pairs of columns $j'$ and $j$ for each $j'<j$, and from the pairs of columns $j$ and $j''$ for each $j''>j$. Note that only the triples that involve the cells within the chain from $u$ to $y$ might have a different contribution to quinv in $\xi$ vs $\sigma$. For readability, we will introduce the following notation. For a filling $\nu$ of $\dg(\lambda)$ and any triple $L$, define $\quinv(\nu;L)$ to be 1 if this triple is a quinv triple in the filling $\nu$, and 0 if either it is not a quinv triple, or if the triple $L$ doesn't exist. Then we may write $\quinv(\nu)=\sum_L \quinv(\nu;L)$, where the sum is over all triples $L$ in $\dg(\lambda)$. 

For $0 \leq \ell \leq m$ and for $j'<j$, define the triple $L_{\ell}^{j'}=\{(i+\ell+1,j'),(i+\ell,j'),(i+\ell,j)\}$, as shown below with $a=\sigma(i+\ell+1,j')$ and $b=\sigma(i+\ell,j')$ (recall that $\xi(i+\ell,j)=k+\ell$ and $\sigma(i+\ell,j)=k+\ell+1$): 

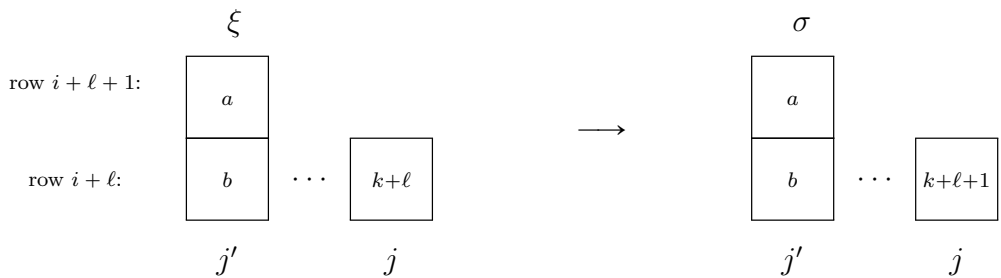
\begin{align*}
& \quad\ \xi & & \quad\ \sigma & \\
\begin{array}{c}
\scriptstyle{\text{row $i+\ell+1$:}} \\[0.8cm]
\scriptstyle{\text{row $i+\ell$:}}
\end{array}
\quad &
\begin{ytableau}
{\scriptstyle a} \\
{\scriptstyle b} & \none[\cdots] & {\scriptstyle k+\ell} \\
\none[j'] & \none & \none[j]
\end{ytableau} &
\qquad
\longrightarrow\qquad
\qquad &
\begin{ytableau}
{\scriptstyle a} \\
{\scriptstyle b} & \none[\cdots] & {\scriptstyle k+\ell+1} \\
\none[j'] & \none & \none[j]
\end{ytableau} &
\end{align*}

Observe that the only way that $(a,b,k+\ell)\in\cQ$ but $(a,b,k+\ell+1)\not\in\cQ$ is if $a=k+\ell+1$ and $b\neq k+\ell$. On the other hand, the only way that $(a,b,k+\ell)\not\in\cQ$ but $(a,b,k+\ell+1)\in\cQ$ is if $a\neq k+\ell+1$ and $b= k+\ell$. Thus we get the following relations for the triples $L_{\ell}^{j'}$. For each $0\leq \ell < m$, we have 
\begin{align*}
\quinv(\xi; L_{\ell}^{j'})-\quinv(\sigma; L_{\ell}^{j'})= 
\delta_{\xi(i+\ell+1,j'), k+\ell+1} - \delta_{\xi(i+\ell,j'),k+\ell}.
\end{align*}
When $\ell=m$, if $b=k+m$, the equality above holds as well, independent of whether or not the cell $(i+m+1,j')$ exists, with one exception: the case when $b=k+m=n$ and $(i+m+1,j')\not\in\dg(\lambda)$. On the other hand, if $b\neq k+m$, then the only time $\quinv(\xi;L^{j'}_m)\neq \quinv(\sigma;L^{j'}_m)$ is when $(i+m+1,j')\not\in\dg(\lambda)$ and $k+m=n$, in which case $(0,b,n)\in\cQ$ but $(0,b,1)\not\in\cQ$. One can check that this simplifies to
\begin{multline}\label{eq:Lm}
\quinv(\xi; L_m^{j'})-\quinv(\sigma; L_m^{j'}) 
=
\delta_{\xi(i+m+1,j'),k+m+1} - \delta_{\xi(i+m,j'), k+m}
+ \delta_{\lambda_{j'}, i+m} \delta_{k+m,n}.
\end{multline}
The expressions for $L^{j'}_{\ell}$ combine into the telescoping sum:
\begin{multline*}
\sum_{\ell=0}^{m} \big(\quinv(\xi; L_{\ell}^{j'})-\quinv(\sigma; L_{\ell}^{j'}) \big)= 
- \delta_{\xi(i,j'),k} + \delta_{\xi(i+m+1,j'), k+m+1} + 
\delta_{\lambda_{j'}, i+m} \delta_{k+m,n}.
\end{multline*}

For $1 \leq \ell \leq m$ and for $j''>j$, define the triple $R_{\ell}^{j''}=\{(i+\ell,j),(i+\ell-1,j),(i+\ell-1,j'')\}$, as shown below with $c=\sigma(i+\ell,j'')$: 

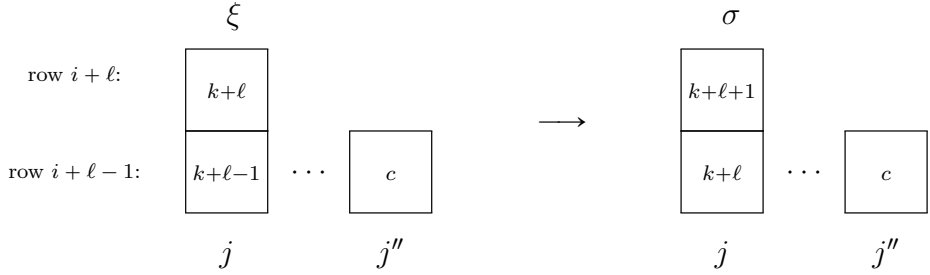
\begin{align*}
& \quad\ \xi & & \quad\ \sigma & \\
\begin{array}{c}
\scriptstyle{\text{row $i+\ell$:}} \\[0.8cm]
\scriptstyle{\text{row $i+\ell-1$:}}
\end{array}
\quad &
\begin{ytableau}
{\scriptstyle k+\ell} \\
{\scriptstyle k+\ell-1} & \none[\cdots] & {\scriptstyle c} \\
\none[j] & \none & \none[j'']
\end{ytableau} &
\longrightarrow\qquad
\quad &
\begin{ytableau}
{\scriptstyle k+\ell+1} \\
{\scriptstyle k+\ell} & \none[\cdots] & {\scriptstyle c} \\
\none[j] & \none & \none[j'']
\end{ytableau} &
\end{align*}
Observe that neither $(k+\ell,k+\ell-1,c)$ or $(k+\ell+1,k+\ell,c)$ is in $\cQ$ for $1 \leq u \leq m$, and so $\quinv(\xi;R_{\ell}^{j''})=\quinv(\sigma;R_{\ell}^{j''})$ for $1 \leq \ell \leq m$. 
The final triples to account for are the following: for each $j''>j$, define the triples 
\[ R_0^{j''}=\{x=(i,j),(i-1,j),(i-1,j'')\}\quad \text{and}\quad R_{m+1}^{j''}=\{(i+m+1,j),y=(i+m,j),(i+m,j'')\}
\]
(if those cells exist). Using that $\xi(i-1,j)\neq k$, we get that $\quinv(\sigma;R_{0}^{j''})=1$ while $\quinv(\xi;R_{0}^{j''})=0$ if and only if $\xi(i-1,j'')=k$, and otherwise they are equal (including the trivial case where this triple doesn't exist). Thus 
\[
\quinv(\xi;R_{0}^{j''})-\quinv(\sigma;R_{0}^{j''}) = -\delta_{\xi(i-1,j''),k}. 
\]
Finally, we notice that since $\xi(i+m+1,j)\neq k+m+1$, if the cell $(i+m+1,j)\in\dg(\lambda)$ exists, we have 
\begin{equation}\label{eq:R}
\quinv(\xi;R_{m+1}^{j''})-\quinv(\sigma;R_{m+1}^{j''})= 
\delta_{\xi(i+m,j''),k+m+1}.
\end{equation}
Otherwise $R_{m+1}^{j''}$ is a degenerate triple. In that case, if $\xi(y)=k+m\neq n$, \eqref{eq:R} still holds. However, if $\xi(y)=n$, then $(0, n, c) \not\in \cQ$ for all $c$, but $(0,1,c)\in\cQ$ unless $c=\xi(i+m,j'')=1$. One can check this simplifies to the following expression for all triples $R^{j''}_{m+1}$:
\begin{equation}\label{eq:Rm}
\quinv(\xi;R_{m+1}^{j''})-\quinv(\sigma;R_{m+1}^{j''})= 
\delta_{\xi(i+m,j''),k+m+1} - \delta_{\lambda_j,i+m} \delta_{k+m,n}.
\end{equation}

Now, since all triples outside of the sets $\{L_{\ell}^{j'}\}_{\substack{j'<j,\\1 \leq \ell \leq m}}$ and $\{R_{\ell}^{j''}\}_{\substack{j''>j,\\0 \leq \ell \leq m+1}}$ are identical in $\xi$ and $\sigma$, we have that $\quinv(\xi)-\quinv(\sigma)$ equals
\begin{align*}
&\hspace{-0.5in}\sum_{j'<j}\sum_{1 \leq \ell \leq m} \left(\quinv(\xi;L_{\ell}^{j'})-\quinv(\sigma;L_{\ell}^{j'})\right) + \sum_{j''>j}\sum_{0 \leq \ell \leq m+1} \left(\quinv(\xi; R_{\ell}^{j''})-\quinv(\sigma; R_{\ell}^{j''})\right) = \nonumber\\
\qquad&\Big(\#\{j'<j: \xi(i+m+1,j')=k+m+1\} -\#\{j'<j: \xi(i,j')=k \}\nonumber \\
&\qquad+ \#\left\{j'<j: \lambda_{j'}=\lambda_j \right\} 
\delta_{\lambda_j,i+m} \delta_{k+m,n}\Big)\\
 &\qquad+ \Big(-\#\left\{j''>j:\sigma(i-1,j'')=k\right\}+\#\left\{j''>j: \xi(i+m,j'')=k+m+1 \right\}\nonumber\\
 &\qquad-\#\{j''>j: \lambda_{j''}=\lambda_j\} 
\delta_{\lambda_j,i+m} \delta_{k+m,n}\Big)\\
 &= \up(\sigma,y)-\down(\xi,x) + 
\delta_{\lambda_j,i+m} \delta_{k+m,n}
\Big(\#\left\{j'<j:\lambda_{j'}=\lambda_j \right\}-\#\left\{j''>j:\lambda_{j''}=\lambda_j \right\}\Big). 
\end{align*}
The last line is \eqref{eq:quinvdiff} since $H_u(\xi)=(i+m,j)$ and $\xi(y)=k+m$ by definition. We now observe that if all parts of $\lambda$ are distinct, the final term in this line vanishes. Thus when $\lambda$ is strict, we obtain \eqref{eq:quinvdistinct}, completing the proof.
\end{proof}

\bibliographystyle{acm}
\bibliography{../../../Macbib}

\end{document}